\documentclass[11pt]{article}
\pdfoutput=1
\usepackage{epigamath}

\setpapertype{A4}

\usepackage[english]{babel}

\usepackage[dvips]{graphicx}     

\title{Finiteness of the space of $n$-cycles for a reduced  $(n-2)$-concave complex space}
\titlemark{Finiteness of the space of $n$-cycles for a reduced  $(n-2)$-concave complex space}
\author{Daniel Barlet}
\authoraddresses{
\authordata{Daniel Barlet}{\firstname{Daniel} \lastname{Barlet}\\
\institution{
Institut Elie Cartan, Universit\'e de Lorraine, CNRS UMR 7502, and Institut Universitaire de France}\\
BP 239, F-54506 Vandoeuvre-l\`es-Nancy cedex, France\\
\email{daniel.barlet@univ-lorraine.fr}}
}
\authormark{D. Barlet}
\date{\vspace{-5ex}} 
\journal{\'Epijournal de G\'eom\'etrie Alg\'ebrique} 
\acceptation{Received by the Editors on June 27, 2016, and in final form on April 3, 2017. \\ Accepted on June 19, 2017.}

\newcommand{\articlereference}{Volume 1 (2017), Article Nr. 5}

\DeclareMathOperator{\E}{\mathbb{E}}
\DeclareMathOperator{\C}{\mathbb{C}}
\DeclareMathOperator{\B}{\mathcal{B}}
\DeclareMathOperator{\Sym}{Sym}
\DeclareMathOperator{\R}{Re}

\usepackage{mathrsfs, yfonts}

\newtheorem{thm}{Theorem}[section]
\newtheorem{defn}[thm]{Definition}
\newtheorem{cor}[thm]{Corollary}
\newtheorem{prop}[thm]{Proposition}
\newtheorem{lemma}[thm]{Lemma}

\newcommand{\parag}[1]{\paragraph{\bf #1.}}

\usepackage[all]{xy}

\begin{document}


\maketitle


\dedication{pour Yum-Tong Siu, avec mon amiti\'e}

\begin{prelims}

\vspace{-0.1cm}

\def\abstractname{Abstract}
\abstract{We show that for $n \geq 2$  the space of closed $n$-cycles in a strongly $(n-2)$-concave complex space has a natural structure of reduced complex space locally of finite dimension and 
 represents the functor ``analytic family of $n$-cycles'' parametrized by Banach analytic sets.}



\vspace{0.05cm}

\languagesection{Fran\c{c}ais}{%

\textbf{Titre. Finitude de l'espace des $n$-cycles pour un espace complexe $(n-2)$-concave r\'eduit}
\commentskip
\textbf{R\'esum\'e.}
Nous montrons que, pour $n\ge 2$, l'espace des $n$-cycles ferm\'es dans un espace complexe fortement $(n-2)$-concave a une structure naturelle d'espace complexe r\'eduit localement de dimension finie et que cet espace repr\'esente le foncteur ``famille analytique de $n$-cycles'' param\'etr\'ee par des ensembles analytiques banachiques.}

\end{prelims}


\newpage

\tableofcontents

\section{Introduction}

The aim of this article is to show that in a reduced strongly $(n-2)$-concave\footnote{Our conventions will be precised below.} complex space $Z$ with $n \geq 2$, the space of closed $n$-cycles  is in a natural way endowed with a structure of a reduced  complex space locally of finite dimension. With its tautological family of $n$-cycles it represents the functor ``analytic family of $n$-cycles in $Z$'' and also the functor ``\textsf{f}-analytic family of $n$-cycles in $Z$'' introduced in \cite{B08} (see also \cite{B13} and \cite{B15}) parametrized by a Banach analytic set.

 This answers a question asked to me  by Y-T. Siu forty years ago.

 I was able to solve this question thanks to the notion of \textsf{f}-analytic family introduced in {\it loc. cit.} and using the space $\mathcal{C}_{n}^{\mathsf{f}}(Z)$ of finite type cycles with its natural topology.

\medskip

 We obtain the following results.

 \begin{thm}\label{germe}
 Let $n \geq 2$ be an integer. Let $Z$ be a strongly $(n-2)$-concave reduced complex space of pure dimension $n+p$  that is to say admitting a  $\mathscr{C}^{2}$ exhaustion function  $\varphi : Z \to ]0, 2]$ which is  strongly $(n-2)$-convex outside  the compact set  $K := \varphi^{-1}([1, 2])$. For any  $\alpha \in ]0, 1[$ and   any $n$-cycle $X_{0}$ in an open neighbourhood of the compact set  $\varphi^{-1}([\alpha, 2])$ there exists $\beta \in ]0, \alpha[$ such that, if $Z_{\beta} := \{ z \in Z \ / \  \varphi(z) >  \beta \}$, the cycle $X_{0}$ extends in a unique way to the open set $Z_{\beta}$ and admits an open neighbourhood $\mathcal{U}$ in the  space $\mathcal{C}_{n}^{\mathsf{f}}(Z_{\beta})$ such that the ringed space defined by $\mathcal{U}$ and the sheaf of holomorphic functions on $\mathcal{U}$ is a (reduced) complex space locally of finite dimension.
   \end{thm}

   Recall that a holomorphic function $h : \mathcal{U} \to \C$ on an open set in $\mathcal{C}_{n}^{\mathsf{f}}(Z_{\beta})$ is a continuous function on $\mathcal{U}$ such that for any holomorphic map $f : S \to \mathcal{U}$ (corresponding to an \textsf{f}-analytic family of $n$-cycles in $Z$, see {\it loc. cit.}) of a Banach analytic set $S$ to $\mathcal{U}$ the composed function $h\circ f$ is holomorphic.

  \begin{thm}\label{global}
  Consider the same situation as in the previous theorem, and let now $X_{0} \in \mathcal{C}_{n}^{\mathsf{f}}(Z)$ be a finite type $n$-cycle in $Z$. Then there exists $\beta \in ]0,1[$ and open neighbourhoods  $\mathcal{V}$  and $\mathcal{U}$ respectively of  $X_{0}$ in $\mathcal{C}_{n}^{\mathsf{f}}(Z)$ and of  $X_{0} \cap Z_{\beta}$ in $\mathcal{C}_{n}^{\mathsf{f}}(Z_{\beta})$ such that the restriction  map
   $$ r : \mathcal{V} \to \mathcal{U} $$
   is well defined and  bi-holomorphic.
    \end{thm}

    We are going to  recall briefly the notion of \textsf{f}-analytic family of finite type $n$-cycles in a complex space $Z$.

    Firstly the notion of an analytic family of $n$-cycles in a reduced complex space $Z$ parametrized by a reduced complex space $S$ is defined as follows (see\footnote{See Chapter 3 Section 4 in \cite{B75} for the case when $S$ is a Banach analytic set.} \cite[Chapter I, p. 33]{B75}   or \cite[Chapter IV, Section 3]{BM1}), using the following notion of an adapted scale (see \cite[Chapter IV, Section 2.1]{BM1}).

    \begin{defn}\label{adapted scale} \rm
    We call $E := (U, B, j)$ a {\bf $n$-scale on a complex space $Z$} when $U$ and $B$ are open relatively compact polydiscs respectively in $\C^{n}$ and $\C^{p}$ and where $j: Z_{E} \to W$ is a closed embedding of an open set $Z_{E}$ in $Z$ into an open neighbourhood $W$ of $\bar U\times \bar B$ in $\C^{n+p}$. The open set $Z_{E}$ is called the {\bf domain} of $E$ and the open set $j^{-1}(U\times B)$ the {\bf center} of $E$. When $X$ is a $n$-cycle in $Z$, the $n$-scale $E$ is {\bf adapted} to $X$ when $\vert X \vert \cap j^{-1}(\bar U\times \partial B) = \emptyset$.
    \end{defn}

    Note that when $E$ is a $n$-scale adapted to a $n$-cycle $X$ in $Z$, the projection of $U\times B$ on $U$ restricted to $j_{*}(X \cap j^{-1}(U\times B))$ gives a finite proper map of degree $k \geq 0$ and the fibers of this map are classified by a holomorphic map $f : U \to \Sym^{k}(B)$. In this case we shall say that $f$ is the {\bf classifying map} of the cycle $X$ in the adapted scale $E$.

    \begin{defn}\label{famille analytique} \rm
    Let $Z$ be a complex space and let $(X_{s})_{s \in S}$ be a family of $n$-cycles in $Z$ parametrized by a reduced complex space $S$. We shall say that this family is {\bf analytic at a point $s_{0} \in S$} if for any  $n$-scale $E := (U, B, j)$ on $Z$ which is adapted to the cycle $X_{s_{0}}$ there exists an open neighbourhood $S_{0}$ of $s_{0}$ in $S$ satisfying the following properties:
    \begin{enumerate}[\rm i)]
    \item For each $s \in S_{0}$ the scale $E$ is adapted to $X_{s}$.
    \item Assume that $k := \deg_{E}(X_{s_{0}})$. Then for each $s \in S_{0}$ we have $\deg_{E}(X_{s}) = k$.
    \item There exists a holomorphic map $f : S_{0}\times U \to \Sym^{k}(B)$ such that for each $s \in S_{0}$ the restriction of $f$ to $\{s_{0}\}\times U$ classifies the cycle $X_{s}$ in the scale $E$.
    \end{enumerate}
    \end{defn}

    It is easy to see that an analytic family of cycles has a ``set theoretic'' graph
     $$\vert G \vert := \{(s, x) \in S \times Z \ / \  x \in \vert X_{s}\vert \}$$
      which is a closed analytic subset in $S \times Z$ and that its projection on $S$ has pure $n$-dimensional fibers (which are the supports of the cycles). When we have an analytic family $(X_{s})_{s \in S}$ and when  the  projection of its graph  $pr : \vert G \vert \to S$ is quasi-proper\footnote{This means, by definition, that for any $s_{0}\in S$ there exists an open neighbourhood $S_{1}$ of $s_{0}$ in $S$ and a compact set $K$ in $\vert G\vert$ such that  any irreducible component of any fiber $pr^{-1}(s)$ for any $s \in S_{1}$ meets $K$.} we shall say that $(X_{s})_{s \in S}$ is a {\bf \textsf{f}-analytic family} of (finite type) $n$-cycles in $Z$ . Of course this condition implies that each cycle $X_{s}$ is a finite type $n$-cycle (it means that each cycle  admits only finitely many irreducible components) but this condition contains this fact in a local uniform manner on $S$.

      In an analogous way, when we have an analytic family $(X_{s})_{s \in S}$  and when  the  projection of its graph  $pr : \vert G \vert \to S$ is proper, we shall say that it is a {\bf proper analytic family of compact cycles in $Z$}. 

\medskip

    The following corollary is of course the main result.

  \begin{cor}\label{important}
   Consider the same situation as in the previous theorems. Then the ringed space given by the sheaf of holomorphic functions on $\mathcal{C}_{n}^{\mathsf{f}}(Z)$ is a reduced complex space locally of finite dimension. Moreover, endowed with its tautological family of  (finite type) $n$-cycles it represents the functor
\begin{eqnarray*} 
\mbox{\rm (reduced complex spaces)} & \to & \mbox{\rm (sets)} \\
  S & \mapsto & \{\mbox{\rm $\mathsf{f}$-analytic family of $n$-cycles in $Z$ parametrized by $S$}\}.
\end{eqnarray*}
     \end{cor}

     \parag{Remark} Let $X$ be a (non empty) irreducible analytic subset of dimension $n$  in a strongly $(n-2)$-concave complex space  $Z$ as in Theorem \ref{germe}. Let $x_{0}\in X$ be a point in $X$ where the supremum of the restriction of the exhaustion  function $\varphi$ to $X$ is obtained. Then the point $x_{0}$ is in the compact set $K =  \varphi^{-1}([1, 2])$ because the Levi form of $\varphi$ at the point $x_{0}$ has at least $n$  non positive eigenvalues. So any (non empty) irreducible analytic subset  of dimension $n$ in $Z$ has to meet the compact set $K$. Then any $n$-cycle in $Z$ is of finite type and any  analytic family of $n$-cycles in $Z$ has a quasi-proper graph so  is a \textsf{f}-analytic family. This implies that, in the previous corollary, the obvious map\footnote{The set $\mathcal{C}_{n}^{loc}(Z)$ is the set of all closed cycles of dimension $n$ in the complex space $Z$; its natural topology which  is associated to the adapted scales is described in  \cite[Chapter IV]{BM1}. In general, the inclusion  map $\mathcal{C}_{n}^{\mathsf{f}}(Z) \to \mathcal{C}_{n}^{loc}(Z)$ is continuous but is not a homeomorphism onto its image. See \cite{B15}  for the comparison between compact sets in these two spaces.} $\mathcal{C}_{n}^{\mathsf{f}}(Z) \to \mathcal{C}_{n}^{loc}(Z)$  is an isomorphism  of ringed spaces and  $\mathcal{C}_{n}^{\mathsf{f}}(Z)$ represents also the functor ``analytic family of $n$-cycles in $Z$''.

     \parag{Question} As it appears in the previous remark, we may expect the same result for $(n-1)$-cycles under our assumption of strong $(n-2)$-concavity. But our way to use Hartogs figures in the present article needs one more positive eigenvalue than one can expect. Is the result also true for $(n-1)$-cycles under our hypothesis? It would probably be interesting, for instance, to have this kind of result for $1$-cycles in a strongly $0$-concave complex space.

\parag{Acknowledgments} I want to thank the referees for many interesting comments and suggestions to improve this article.

 \section{Hartogs figures}

\subsection{Banachization}

 For the analytic extension via Hartogs figures, the use of the Banach spaces $H(\bar U, \C)$  of continuous functions, holomorphic inside on a compact polydisc $\bar U$ is not well adapted. We shall use the Banach space $\B(U, \C)$ of bounded holomorphic functions on $U$ with the ``sup'' norm on $U$. Of course, $H(\bar U, \C)$ is a closed Banach subspace in $\B(U, \C)$.

 \begin{prop}\label{banach borne}
 Let $U$ be a relatively compact polydisc in $\C^{n}$ and let $S$  be a Banach analytic set. Let $F : S \to \B(U, \C)$ be a holomorphic map. Then the corresponding   function $f : S \times U \to \C$ defined by  $f(s,t) := F(s)[t]$ for $(s,t)$ in $ S\times U$ is holomorphic (and locally  on $S$ uniformely bounded on $U$).

 Conversely, if we have a holomorphic function $f : S \times U \to \C$ and an open  polydisc $U''\subset\subset U$, the associated map $F : S \to \B(U'', \C)$, defined for $(s,t)$ in  $S\times U''$ by  $F(s)[t] := f(s, t)$, is holomorphic.
  \end{prop}

\begin{proof} 
The evaluation function $ev : \B(U, \C) \times U \to \C$ is holomorphic as one may easily see by differentiation. Then the function $f$  associated to $F$ is the composition of the holomorphic maps $F\times id_{U}$ and $ev$. So it is holomorphic.

The converse is consequence of the linear isometric inclusion of $H(\bar U'', \C)$ in $\mathcal{B}(U'', \C)$ as  $F'' : S \to H(\bar U'', \C)$ is holomorphic as soon as $f$ is holomorphic (see  \cite{B75} or \cite{BM2}).
\hfill $\Box$
\end{proof}

 \subsection{$n$-Hartogs figure on a complex space}

\label{2.2}

For $\alpha \in (\mathbb{R}_{+}^{*})^2$ let $M(\alpha)$ be the open set in $\mathbb{C}^2$ defined by
\begin{align*}
& M(\alpha) : = M^P(\alpha) \cup M^C(\alpha) \quad {\rm with}\\
& M^P(\alpha) : = \big\{ \vert t_1 - \alpha_1/2\vert < \alpha_1/4, \quad \vert t_2\vert < \alpha_2 \big\} \\
& M^C(\alpha) : = \big\{ \vert t_1 \vert < \alpha_1, \quad \alpha_2/2 < \vert t_2\vert < \alpha_2 \big\}
 \end{align*}
 and let also $\mathcal{M}(\alpha) : = \big\{(t_1,  t_2) \in \mathbb{C}^2 \ / \  \vert t_i\vert < \alpha_i, \ i = 1, 2 \big\}$.

 For  $\varepsilon > 0$ small enough, we define
 \begin{align*}
& M(\alpha)^{\varepsilon} : = M^P(\alpha)^{\varepsilon}  \cup M^C(\alpha)^{\varepsilon}  \quad {\rm with}\\
& M^P(\alpha)^{\varepsilon}  : = \big\{ \vert t_1 - \alpha_1/2\vert < \alpha_1/4- \varepsilon/4, \quad \vert t_2\vert < \alpha_2 - \varepsilon \big\} \\
& M^C(\alpha)^{\varepsilon}  : = \big\{ \vert t_1 \vert < \alpha_1-\varepsilon, \quad \alpha_2/2 + \varepsilon/2 < \vert t_2\vert < \alpha_2- \varepsilon \big\}
 \end{align*}
and also $\mathcal{M}(\alpha)^{\varepsilon}  : = \big\{(t_1,  t_2) \in \mathbb{C}^2 \ / \  \vert t_i\vert < \alpha_i- \varepsilon, \ i = 1, 2 \big\}$.

\medskip

For a polydisc of radius $R$ in $\C^{m}$ we shall denote by $P^{\varepsilon}$ the polydisc with same center and radius  $R - \varepsilon$ for $0 < \varepsilon < R$.

  \begin{defn}\label{$n-$Marmite} \rm
  Let  $n \geq 2$ and  $p \geq 1$ two integers and let $\Delta \subset\subset \Delta'$ be two open sets in a reduced complex space $Z$. We shall call $ \mathcal{H} := (\mathcal{M}, M, B, j)$ a   {\bf  $n$-Hartogs figure in $Z$ relative to the boundary of $\Delta$}, the following  data
  \begin{itemize}
  \item  an embedding $j$ of an open set $Z'$ in $\Delta'$ into an open set in $\C^{n+p}$,
  \item  open sets $\mathcal{M} \subset\subset \Delta'$ and $M \subset\subset \Delta$ relatively compact in $Z'$,
  \item  a relatively compact polydisc $B$ in $\C^{p}$
  \end{itemize}
  such that there exists $\alpha \in (\mathbb{R}_{+}^{*})^2$ and a relatively compact polydisc $V$ in $\C^{n-2}$  with the following conditions:
 \begin{enumerate}[\rm i)]
 \item The map $j$ induces a closed embedding of $\mathcal{M}$ in $\mathcal{M}(\alpha)\times V\times B$.
 \item  The map  $j$ induces a closed embedding of $M$ in $M(\alpha)\times V\times B$.
 \item We have  $j^{-1}\big(\bar{ \mathcal{M}}(\alpha)\times \bar V\times \partial B\big) \subset  \Delta$.\\
  \end{enumerate}
 \end{defn}

\begin{defn}\label{adapte}
\rm
  Let  $n \geq 2$ and  $p \geq 1$ two integers and let $\Delta \subset\subset \Delta'$ be two open sets in a reduced complex space $Z$. Let $\mathcal{H} = (\mathcal{M}, M, B, j)$ be a $n$-Hartogs figure in $Z$ relative to the boundary of $\Delta$ and let  $X_{0}$ be a $n$-cycle in $\Delta$. We shall say that  {\bf $\mathcal{H}$ is adapted to  $X_{0}$} when the following condition is satisfied :
  \begin{equation*}
  j^{-1}(\bar{\mathcal{M}}(\alpha)\times \bar V\times \partial B) \cap \vert X_{0}\vert = \emptyset . \tag{@}
  \end{equation*}
\end{defn}

 \parag{Remarks} \begin{enumerate}
 \item Let $\mathcal{H}$ be a  $n$-Hartogs figure in $Z$ relative to the boundary of $\Delta$. If the open set $\Delta_{1} \subset\subset \Delta'$ has a boundary $\partial \Delta_{1}$ near enough to $\partial \Delta$, then $\mathcal{H}$ is again a $n$-Hartogs figure relative to the boundary of  $\Delta_{1}$. For instance, if  $\Delta := \{ \varphi > 0 \}$ where $\varphi$ is a continuous proper function on $Z$, we may choose $\Delta_{1} := \{ \varphi > \varepsilon\}$ for $\varepsilon > 0$ small enough.
 \item Note that the $n$-scale $E_{\mathcal{H}} := (M(\alpha)\times V, B, j)$ on $\Delta$ associated to $\mathcal{H}$ is adapted to  $X_{0} $ as soon as the $n$-Hartogs figure  $\mathcal{H}$ is adapted to $X_{0}$.
\item If $\tilde{X}_{0}$ is a $n$-cycle in $\Delta'$ such that its restriction to $\Delta$ is equal to $X_{0}$, the $n$-scale $E_{\tilde{\mathcal{H}}} := (\mathcal{M}(\alpha)\times V, B, j)$ on $\Delta'$ is adapted to  $\tilde{X}_{0}$ if and only if the $n$-Hartogs figure $\mathcal{H}$ is adapted to $X_{0}$.
  \end{enumerate}

\medskip

  Note that the $n$-scale  $E_{\tilde{\mathcal{H}}}$ is not a $n$-scale on $\Delta$ although the subset $\bar{\mathcal{M}}(\alpha)\times \bar V\times \partial B$  is contained in $\Delta$.

 \begin{defn}\label{homotheties} \rm
 In the situation above we define, for $\varepsilon >  0$ small enough, the $n$-Hartogs figure $\mathcal{H}^{\varepsilon}$ on $\Delta$ as follows :
  $$ \mathcal{H}^{\varepsilon} := (\mathcal{M}^{\varepsilon}, M^{\varepsilon}, B, j)$$
  where we use the notations
   $$\mathcal{M}^{\varepsilon} := j^{-1}(\mathcal{M}(\alpha)^{\varepsilon}\times V^{\varepsilon}\times B) \quad {\rm and~also} \quad M^{\varepsilon}:= j^{-1}(M(\alpha)^{\varepsilon}\times V^{\varepsilon}\times B).$$
 \end{defn}

 It is obvious to see that when $\mathcal{H}$ is a $n$-Hartogs figure relative to the boundary of $\Delta$, then  $\mathcal{H}^{\varepsilon}$ is again a $n$-Hartogs figure relative to the boundary of $\Delta$ for all $\varepsilon > 0$ small enough.

 Moreover, if $\mathcal{H}$  is adapted to the $n$-cycle $X_{0}$ of $\Delta$,  the same is true for $\mathcal{H}^{\varepsilon}$ for all $\varepsilon > 0$ small enough.

  \begin{lemma}\label{banach}
  Let $V$ be a relatively compact open polydisc in $\C^{q}$. The restriction map
   $$ res : \B(\mathcal{M}(\alpha)\times V, \C) \to \B(M(\alpha)\times V, \C)$$
   is a linear isometry of Banach spaces.
    \end{lemma}

    Note that the restriction map
     $$ res : \B(\mathcal{M}(\alpha)^{\varepsilon}\times V^{\varepsilon}, \C) \to \B(M(\alpha)^{\varepsilon}\times V^{\varepsilon}, \C)$$
   induces also an isometry  for all $\varepsilon > 0 $ small enough.

\begin{proof} 
Let $f(v, t_{1}, t_{2}) := \sum_{m \in \mathcal{Z}}  a_{m}(v, t_{1}).t_{2}^{m}$ the Laurent expansion of the holomorphic function $f : M^{C}(\alpha)\times V \to \C$. The holomorphic functions $a_{m}$, $m \in \mathbb{Z}$ on the product of $V$ by the disc $\{ \vert t_{1}\vert < \alpha_{1} \}$ are given by the formula
      $$ a_{m}(v, t_{1}) := \frac{1}{2i\pi}.\int_{\vert z \vert = r} f(v, t_{1},z).\frac{dz}{z^{m+1}}  \quad {\rm with} \quad r \in ]\alpha_{2}/2, \alpha_{2}[.$$
      As the holomorphy of $f$ on $M^{P}(\alpha)\times V$ implies that $a_{m} \equiv 0$ for each negative $m$ on the open set $\{ \vert t_1 - \alpha_1/2\vert < \alpha_1/4\}\times V$, we conclude that the functions $a_{m} $ are identically zero for $m < 0$ and so $f$ is holomorphic on  $\mathcal{M}(\alpha)\times V$. This shows that the restriction map \ $res$ \ is bijective (and also it is linear continuous) between the  two Fr\'echet spaces $\mathcal{O}(\mathcal{M}(\alpha)\times V)$ and $\mathcal{O}(M(\alpha)\times V)$; so it is an isomorphism of Fr\'echet spaces.

      Let us show that if $f$ is in $\B(\mathcal{M}(\alpha)\times V, \C)$, then $res(f)$, which belongs to the space  $\B(M(\alpha)\times V, \C)$, has the same  ``sup'' norm.  For this purpose fix $\varepsilon > 0$ small enough. As $\bar{\mathcal{M}}(\alpha)^{\varepsilon}\times \bar V^{\varepsilon}$ is a compact polydisc in $\mathcal{M}(\alpha)\times V$, the maximum of $f$ on this compact is obtained at some point $z$ in the distinguish boundary of it. But as $z$ is also in the boundary of $M^{C}(\alpha)^{\varepsilon}\times \bar V^{\varepsilon}$, the desired equality follows.\\
      Conversely, if $g$ is in $\B(M(\alpha)\times V, \C)$, its analytic extension $f$ to $\mathcal{M}(\alpha)\times V$ will be bounded on the boundary of  $\mathcal{M}(\alpha)^{\varepsilon}\times V^{\varepsilon}$ by the sup of $g$ on $M^{C}(\alpha)^{\varepsilon}\times V^{\varepsilon}$.  So we obtain the equality of the ``sup'' norms for $g$ and $f$ respectively on $M(\alpha)\times V$ and $\mathcal{M}(\alpha)\times V$.
\hfill $\Box$
\end{proof}

      \parag{The Banach analytic set  $\mathcal{B}(U, \Sym^{k}(\C^{p}))$} Recall that if $p \geq 1$ and $k \geq 1$  are   integers, there exists a closed embedding (in fact given by a polynomial map)  of  $\Sym^{k}(\C^{p}) := (\C^{p})^{k}\big/\mathfrak{S}_{k}$ into the vector space $E(k) := \bigoplus_{h=1}^{k} S^{h}(\C^{p})$ given by the elementary tensorial symetric functions\footnote{See for instance \cite[Chapter I, \S 4]{BM1}.}. If $U$ is an open relatively compact polydisc in $\C^{n}$, the subset  $\mathcal{B}(U, \Sym^{k}(\C^{p}))$ is closed and Banach analytic in the Banach space $\mathcal{B}(U, E(k))$. Indeed, if $Q : E(k) \to \C^{N}$ is a polynomial map such that $Q^{-1}(0) = \Sym^{k}(\C^{p})$, then the holomorphic map
       $$ \mathcal{Q} : \mathcal{B}(U, E(k)) \to \mathcal{B}(U, \C^{N}) \quad {\rm defined~by} \quad  f \mapsto Q\circ f $$
       satisfies $ \mathcal{Q}^{-1}(0) = \mathcal{B}(U, \Sym^{k}(\C^{p}))$.

       Nevertheless, be aware that for an open set $\Omega$ in  $E(k)$  the subset $\mathcal{B}(U, \Omega)$ of elements in $\mathcal{B}(U, E(k))$ taking their values in $\Omega$ is  not, in general,  open in $\mathcal{B}(U, E(k))$; so, for an open polydisc $B \subset\subset \C^{p}$,  the subset  $ \mathcal{B}(U, \Sym^{k}(B))$ is not open in $\mathcal{B}(U, \Sym^{k}(\C^{p}))$ in general.

\parag{Remark} The obvious map  $H(\bar U, E(k)) \to \mathcal{B}(U, E(k))$ is a closed (linear) isometry and induces a holomorphic inclusion map
 $$ i_{U} : H(\bar U, \Sym^{k}(\C^{p})) \hookrightarrow \mathcal{B}(U, \Sym^{k}(\C^{p})) $$
 and for all $\varepsilon > 0$ the restriction
  $$ r :   \mathcal{B}(U, E(k)) \to H(\bar U^{\varepsilon}, E(k)) $$
  is a (linear and  continuous) compact map which induces a holomorphic restriction map
    $$ \mathcal{B}(U, \Sym^{k}(\C^{p})) \to H(\bar U^{\varepsilon}, \Sym^{k}(\C^{p})).$$
    This remark will allow us to use Lemma \ref{banach} with Banach analytic sets like $H(\bar U, \Sym^{k}(B))$.

  \parag{Notations} Let  $k \in \mathbb{N}$ and let  $U' \subset \subset U \subset \subset \C^{n}$ and $B \subset\subset \C^{p}$ be polydiscs. We shall note $\Sigma_{U, U'}(k)$ the Banach analytic set classifying the couples of an element in $H(\bar U, \Sym^{k}(B))$ with its isotropy data on $\bar U'$. Recall that for a holomorphic map $f : U \to \Sym^{k}(B)$ the isotropy data on the polydisc $U'$ is the map
  $$ T(f) : U' \to F \otimes E'$$
  where
    $$F := \bigoplus_{i}\Big( L(\Lambda^{i}(\C^{n}), \Lambda^{i}(\C^{p}))\Big) \quad {\rm and} \quad E' := \bigoplus_{m=0}^{k-1} S^{m}(\C^{p}).$$

   It corresponds to  the collection of holomorphic maps
    $$T^{i}_{m}(f) : U' \to L(\Lambda^{i}(\C^{n}), \Lambda^{i}(\C^{p}))\otimes S^{m}(\C^{p})$$
     for all $i \in [1,\mathrm{min}(n, p)]$ and $m \in [0, k-1]$ which are given, near a point in $U'$ where the multiform graph $X_{f}$ associated to $f$ has local branches $f_{1}, \dots, f_{k}$, by the formula
  $$ T^{i}_{m}(f)(t) := \sum_{j=1}^{k}  \Lambda^{i}(Df_{j}(t))\otimes f_{j}(t)^{m} .$$
  These maps are always holomorphic on all $U$ and determine the trace map for the projection $\pi : X_{f} \to U$ of the holomorphic differential forms on $U \times B$.

  The  subset
  $$ \Sigma_{U, U'}(k) \subset  H(\bar U, \Sym^{k}(B)) \times H(\bar U', F\otimes E')$$
   is defined as the graph of the map $f \mapsto T(f) :=  \bigoplus_{i, m} T^{i}_{m}(f)$ which is not holomorphic in general. Nevertheless it  is a Banach analytic subset and its natural  projection
     $$ \Sigma_{U, U'}(k) \to H(\bar U, \Sym^{k}(B))$$
      is a holomorphic homeomorphism (see \cite[Chapter III, Proposition 2, p.\,81]{B75} or \cite[Chapter~V]{BM2}).

      The important point which motivates the introduction of this Banach analytic subset is the fact that, when $S$ is a reduced complex space, an analytic family of multiform graphs given by a holomorphic map
      $f : S\times U \to \Sym^{k}(B)$ will give an analytic family of cycles in $U\times B$ parametrized by $S$ if and only if the corresponding isotropy data (given by the maps $T^{i}_{m}(f_{/S})$ on $S\times U$) are holomorphic on $S\times U$. This is the isotropy condition ; see \cite[Chapter II]{B75} or  \cite[Chapter IV, Section 5]{BM1}.

\medskip

  Our next result is the main tool for performing the analytic extension of $n$-cycles near a $(n-2)$-concave boundary.

   \begin{prop}\label{prlgt isotrope}
Consider the open sets $\mathcal{M}(\alpha)\times V$ and $ M(\alpha)\times V$ in $\C^{n}$. The inverse of the restriction map is a holomorphic isomorphism of analytic extension
 $$  prlgt : \mathcal{B}(M(\alpha)\times V, \Sym^{k}(\C^{p})) \to  \mathcal{B}(\mathcal{M}(\alpha)\times V, \Sym^{k}(\C^{p})). $$
 Composed with the restriction to the compact set $\bar{\mathcal{M}}(\alpha)^{\varepsilon} \times \bar V^{\varepsilon} $ it sends the subset  $H(\bar M(\alpha)\times \bar V, \Sym^{k}(B))$ into  $H(\bar{\mathcal{M}}(\alpha)^{\varepsilon}\times \bar V^{\varepsilon},  \Sym^{k}(B))$ for $\varepsilon > 0$ small enough.

 Moreover, this holomorphic map induces a holomorphic map, again for $\varepsilon > 0$ small enough
  $$ \Sigma_{M(\alpha)\times V, \ M(\alpha)^{\varepsilon/3}\times V^{\varepsilon/3}}(k) \longrightarrow \Sigma_{\mathcal{M}(\alpha)^{\varepsilon/3}\times V^{\varepsilon/3}, \
 \mathcal{M}(\alpha)^{2\varepsilon/3}\times V^{2\varepsilon/3}}(k) $$
which factorizes the restriction map
 $$\Sigma_{\mathcal{M}(\alpha)\times V,\ \mathcal{M}(\alpha)^{\varepsilon/3}\times V^{\varepsilon/3}}(k) \to \Sigma_{\mathcal{M}(\alpha)^{\varepsilon/3}\times V^{\varepsilon/3},\
  \mathcal{M}(\alpha)^{2\varepsilon/3}\times V^{2\varepsilon/3}}(k)$$
  through the restriction
   $$\Sigma_{\mathcal{M}(\alpha)\times V,\  \mathcal{M}(\alpha)^{\varepsilon/3}\times V^{\varepsilon/3}}(k) \to \Sigma_{M(\alpha)\times V,\  M(\alpha)^{\varepsilon/3}\times V^{\varepsilon/3}}(k).$$
\end{prop}

\begin{proof} Lemma  \ref{banach} gives that the map
$$ prlgt : \mathcal{B}(M(\alpha)\times V, E(k)) \to  \mathcal{B}(\mathcal{M}(\alpha)\times V, E(k)) $$
is an isometry of Banach spaces. It is clear that its inverse sends $\mathcal{B}(\mathcal{M}(\alpha)\times V, \Sym^{k}(\C^{p})) $ into the Banach analytic subset $\mathcal{B}(M(\alpha)\times V, \Sym^{k}(\C^{p}))$ and that if  the map $f \in \mathcal{B}(\mathcal{M}(\alpha)\times V, \Sym^{k}(\C^{p})) $ takes values in $\Sym^{k}(B)$, the same is true for its restriction. This proves the first part of the proposition.

In order to prove the second part, it is enough to show that a holomorphic map of a Banach analytic set $S$ with values in the subset
  $$H(\bar M(\alpha)\times \bar V, \Sym^{k}(B))$$
  which is isotropic on the product of $S$ with any relatively compact subset in the open set  $M(\alpha)^{\varepsilon/3}\times V^{\varepsilon/3}$ will have  an analytic extension which will be isotropic on any relatively compact open set in  $\mathcal{M}(\alpha)^{\varepsilon/3}\times V^{\varepsilon/3}$. So it will be isotropic on the closure of the open set $\mathcal{M}(\alpha)^{2\varepsilon/3}\times V^{2\varepsilon/3}$.
\hfill $\Box$
\end{proof}

 \begin{prop}\label{Hart.}
Let $n \geq 2$ and $p \geq 1$ be integers and let $U_1 \times B_1$ the product of two polydiscs with centers $0$  respectively in  $\mathbb{C}^{n}$ and $\mathbb{C}^{p}$. Denote by $(t_1, \cdots, t_n, x_1, \cdots, x_p)$ coordinates on $U_1 \times B_1$.  Let $\varphi $ be a real valued function of class $\mathscr{C}^2$ on $U_1 \times B_1$, such that
 \begin{equation*}
  \varphi(t,x) = \R(t_1) + \sum_{i=1}^n \rho_i.\vert t_i \vert^2 + \sum_{j=1}^p \sigma_j.\vert x_j \vert ^2 + o(\vert\vert (t, x)\vert\vert^2) \tag{@@}
  \end{equation*}
  where the real numbers $\rho_2, \sigma_1, \cdots, \sigma_p$ are positive (so $\varphi$ is $(n-2)$-convex near $(0, 0)$ and $d\varphi_{0,0}\not= 0$).

  Let  $\Delta$ be the open set $\{\varphi  > 0\} $ in $U_1 \times B_1$ and let  $\Delta'$ be an open neighbourhood of the compact set $\bar{\Delta}$ in $\C^{n+p}$. Let $X_0$  be a closed analytic subset of pure dimension $n$ in $\Delta'$ such that each irreducible component of $X_{0}$ meets $\Delta$ and such that
   \begin{equation*}
 \vert X_0\vert  \cap \, \{t_1= \cdots = t_n = 0 \} \subset \{ 0 \} .\tag{*}
  \end{equation*}
  Then there exists $\alpha \in (\mathbb{R}_{+}^{*})^2$  and polydiscs $V$ and  $B \subset\subset B_{1}$ with centers $0$ respectively in $\mathbb{C}^{n-2}$ and $\mathbb{C}^p$ such that the following conditions are satisfied :
\begin{enumerate}
 \item[\rm 1.] $\mathcal{M}(\alpha) \times V \times B \subset \subset \Delta' $;
 \item[\rm 2.] $ M(\alpha) \times V  \times B  \subset\subset \Delta $;
 \item[\rm 3.] $\overline{ \mathcal{M}(\alpha)}\times \bar{V} \times \partial B  \subset \Delta $;
 \item[\rm 4.] $ \vert X_0 \vert \cap \, (\overline{\mathcal{M}(\alpha)} \times \bar{V}\times  \partial B ) = \emptyset $ (this implies $ \vert X_0 \vert \cap \, (\overline{M(\alpha)} \times \bar{V}\times  \partial B ) = \emptyset ).$
\end{enumerate}
 \end{prop}

\begin{proof} Choose the polydisc $B \subset\subset B_{1}$  small enough in order that we have
 $$X_0 \cap (\{0\}\times \bar{B}) \subset  \{0\} \quad  {\rm and} \quad\{0\} \times \partial B \subset \Delta .$$
 This is possible as we have  $\vert X_0\vert \cap \, \{t_1= \cdots = t_n = 0 \} \subset \{0\} $  and as   $\varphi $  is positive on a small enough  punctured neighbourhood of the origin in the $p$-plane  $\{t_1= \cdots = t_n = 0 \} \times \mathbb{C}^p$. So we shall have
 $$\vert X_0\vert \cap (\bar{W} \times \partial B) = \emptyset \quad {\rm and} \quad \bar{W}\times \partial B \subset \Delta$$
 for any small enough  open neighbourhood $W$ of the origin in  $\mathbb{C}^n$.  A immediate consequence is that Conditions 1, 3 and 4 will be satisfied as soon as  $\alpha$ and $V$ are small enough.

In order to check Condition 2, let us remark first that, up to choosing the real numbers  $\rho'_1$ and $\rho'_2$ such that ${\rho'}_1 > \vert\rho_1\vert, \ {\rho'}_2\in ]0,\rho_2[$ and  $r  >  \sup_{i\geq 3} \vert \rho_i \vert $, we obtain on $W \times B$  chosen small enough
 \begin{equation*}
  \varphi(t,x) \geq \R(t_1) - \rho'_1.\vert t_1\vert^2 + \rho'_2.\vert t_2 \vert^2 - r.(\sum_3^n \vert t_i \vert^2) \tag{**}
  \end{equation*}
  with strict inequality as soon as $x \not= 0$. Then for
   $$V_{\varepsilon} = \{(t_3, \cdots, t_n) \ /  \ \sum_3^n \vert t_i \vert^2 < \varepsilon^2 \}$$
   the following inequalities hold
    \begin{align*}
 \varphi(t,x) \geq & \frac{1}{4}\alpha_1 - {\rho'}_1.{\alpha_1}^2  - r.\varepsilon^2 \quad {\rm on} \quad \overline{M^P(\alpha)\times V_{\varepsilon}  \times B} \tag{1} \\
  \varphi(t,x) \geq & - \alpha_1 - {\rho'}_1.{\alpha_1}^2 +  \frac{1}{4}{\rho'}_2.{\alpha_2}^2 - r.\varepsilon^2 \quad {\rm on} \quad \overline{M^C(\alpha)\times V_{\varepsilon}  \times B} \tag{2}
  \end{align*}
  for $\alpha$ and $\varepsilon$ small enough in order that $\mathcal{M}(\alpha)\times V_{\varepsilon} $ is contained in $W$.

This allows to fix  $\alpha_{1}, \alpha_{2}$ and $\varepsilon$.

Now we shall choose $\alpha_{1}$ and  $\varepsilon$ smaller in order to satisfy the following conditions :
 \begin{equation*}
  8 \alpha_1 < {\rho'}_2.{\alpha_2}^2 \ , \quad  \alpha_1 < \frac{1}{8{\rho'}_1} \quad {\rm and} \quad \varepsilon^2 < \frac{1}{8r}\alpha_1. \tag{3}
  \end{equation*}
  To obtain $M^P(\alpha) \times V_{\varepsilon} \times B \subset \subset \Delta $ it is enough to show that on $M^P(\alpha)\times V_{\varepsilon} \times B$ we have, if we let $\alpha_1 = u.\alpha_2 $ and  $ \varepsilon^2 = v.\alpha_1 = uv.\alpha_2$
    $$ \frac{1}{4} > {\rho'}_1.\alpha_1 + r.v .$$
    Indeed, as we assumed $  {\rho'}_1.\alpha_1 < \frac{1}{8}$ and $r.\varepsilon^2 < \frac{1}{8}.\alpha_1$ (so $ r.v  <  1/8$) the first condition holds.

    In order to satisfy $M^C(\alpha) \times V_{\varepsilon} \times B \subset \subset \Delta $ it is enough to show that on   $ M^C(\alpha)\times V_{\varepsilon} \times B$  we have
     $$ \frac{1}{4}{\rho'}_2.\alpha_2 >  u + {\rho'}_1u^2\alpha_2 + r.uv .$$
     But our condition implies
      \begin{equation*}
  {\rho'}_1u^2\alpha_2 < \frac{1}{8}.u \quad \quad r.uv <  \frac{1}{8}.u
  \end{equation*}
which gives $u + {\rho'}_1u^2.\alpha_2 + r.uv <  2.u $. The condition $8 \alpha_1 <  {\rho'}_2.{\alpha_2}^2 $ which implies $2u < \frac{1}{4}\rho'_{2}.\alpha_{2}$, allows to conclude.
\hfill $\Box$
\end{proof}

\parag{Remarks} \begin{enumerate}
\item We only used Condition  $(^*)$ for  $X_0$  and Inequality $(^{**})$ for $\varphi$  in a neighbourhood of the origin in the proof above.
\item Sufficient conditions on $\varphi \in \mathscr{C}^2$  to satisfy ($@@$)  are:
\begin{enumerate}[i)]
\item The origin is not a critical point of $\varphi$.
\item The Levi form of $\varphi$  at $0$ has, at most, $(n-2)$  non positive eigenvalues in the complex tangent hyperplane to the real hypersurface $\{ \varphi(z) = 0\}$; the existence of  real function  $\varphi \in \mathscr{C}^{2} $ such that $\Delta = \{ \varphi > 0 \}$ and satisfying these two conditions is equivalent to the fact that the open set $\Delta$ has a strongly $(n-2)$-concave  smooth boundary near the origin (see Definition \ref{q-concave} given below). Indeed, if  $\varphi$ is not critical at $0$ and has a Levi form at $0$ with, at most, $(n-2)$ non positive  eigenvalues in the complex hyperplane tangent to the real hypersurface $\{ \varphi(z) = 0\}$, its order $2$ Taylor expansion at the origin is written, in suitable local holomorphic  coordinates $(\tau, x)$
 $$ \varphi(\tau,x) = \R(\tau_1) + \R(Q(\tau, x)) + \sum_{i=1}^n \rho_i.\vert \tau_i \vert^2 + \sum_{j=1}^p \sigma_j.\vert x_j \vert ^2 + o(\vert\vert (\tau, x)\vert\vert^2) $$
 where $Q$ is a holomorphic homogeneous degree $2$ polynomial and where the real numbers $\rho_{2}$ and $\sigma_{j}, j \in [1,p]$ are positive. Define new local holomorphic coordinates
 $$ t_{1} := \tau_{1} + Q(\tau, x), t_{i} := \tau_{i} \quad {\rm for} \ i\in [2,n] \quad {\rm and} \quad x_{j} := x_{j} \quad {\rm for} \ j \in [1,p].$$
 Then we obtain $(@@)$.
 \end{enumerate}
 \item Condition  $(^{*})$ implies that $X_0$ has no local irreducible component at $0$ contained in the hyperplane $\{t_1 = 0\}$. In fact, as the coordinate $t_{1}$ is choosen in order to suppress the real part of the holomorphic homogeneous degree $2$ term  in the  order $2$  Taylor expansion of $\varphi$ at $0$ (see the previous remark), we want that no local irreducible component of $X_{0}$ at the origin  is contained in the complex hypersurface $\tau_{1} + Q(\tau, x) = 0$ locally defined near $0$ for $\varphi$ given. Then, as soon as  the restriction  $\varphi_{\vert X_{0}}$ has not $0$ as a critical point, Condition $(^{*})$ could be realized when the Levi form of $\varphi$ has at most $(n-2)$ non positive eigenvalues on the complex tangent hyperplane at the origin of  the hypersurface $\{ \varphi = 0\}$ .
 \item One may easily see that under our hypothesis, the cycle $X_{0}$ meets the open set $\Delta$ when it contains $0$. Indeed, the analytic subset
  $$ \{t_1 = t_3 = \cdots = t_n = 0 \} \cap \vert X_0 \vert $$
   is nonempty,  has dimension at least $1$ and  meets the complement of $\Delta$ only at the origin.

 Of course, assuming that $0 \in X_{0}$, the proposition shows that, in fact,  $X_0$ contains a branched covering of degree $k \geq 1 $ of $M^P(\alpha)\times V_{\varepsilon}  $ inside
  $$M^P(\alpha)\times V_{\varepsilon}  \times B \subset\subset  \Delta.$$
 \item In the situation of Proposition  \ref{Hart.}, for any continuous family  $(X_s)_{s \in S} $  of $n$-cycles in $\Delta$ parametrized by a Banach analytic set $S$ such that $X_{s_0} = X_0 \cap \Delta$, there exists an open neighbourhood $S'$ of $s_0$ in $S$, such that for each $s \in S'$ Condition 4 remains true after analytic extension of the cycles (see Proposition \ref{prlgt isotrope}), because, thanks to Condition 3, $\overline{\mathcal{M}(\alpha)}\times \bar{V} \times \partial B$ is a compact subset in   $\Delta$.
 \end{enumerate}

 \subsection{$q$-concave open sets.}

 \begin{defn}\label{fortement convexe} 
\rm
Let  $\varphi : U \to \mathbb{R}$  be real valued $\mathscr{C}^{2}$ function on an open set $U$ in $\C^{N}$. We shall say that $\varphi$  is {\bf strongly  $q$-convex} when its Levi form at each point of $U$ has at most $q$ non positive eigenvalues.
 \end{defn}

So, with this definition a strongly $0$-convex function is a strongly plurisubharmonic function.

\begin{defn}\label{ambiant}
Let  $\varphi : Z \to \mathbb{R}$  a real valued $\mathscr{C}^{2}$  function on a reduced complex space $Z$. We shall say that  $\varphi$ is {\bf strongly $q$-convex} if locally near each point of $Z$ it can be induced by a $\mathscr{C}^{2}$ strongly $q$-convex  function in a local embedding in an open set of an affine space.
\end{defn}

  Remark that a strongly $q$-convex function on an irreducible complex space of dimension at least equal to $q+1$ has no local maximum because there exists at any point a germ of curve on which the restriction of  $\varphi$ is strongly p.s.h.

\begin{defn}\label{q-concave} \rm
Let $Z$ be a reduced complex space and let $\Delta$ be a relatively compact open set in $Z$. We shall say that $\Delta$ has a {\bf smooth $\mathscr{C}^2$ boundary} when for each point $z$  in  $\partial \Delta$ there exists a local holomorphic embedding $j : W \to U$  of an open neighbourhood $W$ of $z$ in an open set $U$ of the Zariski tangent space of $Z$ at $z$  and an open set $D$ with smooth $\mathscr{C}^2$ boundary in $U$ such that $W \cap j^{-1}(\partial D) =  W \cap \partial \Delta$.

We shall say that the open set  $\Delta \subset Z$ with smooth $\mathscr{C}^2$ boundary is {\bf strongly $q$-concave at a point  $z \in \partial\Delta$} if, in some local holomorphic embedding $j : W \to U$  of $Z$ around $z$ as above, one can define $\Delta$ in  $W$ as the subset   $\{j\circ \varphi > 0 \}\cap W$ where $\varphi$ is a real valued  $\mathscr{C}^2$ function on $U$  such that
\begin{enumerate}
\item  $d\varphi_{j(z)} \not= 0 $ on the tangent space $T_{U,j(z)}$ of $U$ at $j(z)$.
\item  The restriction of the Levi form at $j(z)$ of  $\varphi$ to the complex hyperplane tangent at $j(z)$ to the real hypersurface $\{\varphi(x) = \varphi(j(z))\}$ in $U$  has at most  $q$ non positive eigenvalues.
\end{enumerate}
We shall say that  $\Delta$ is {\bf strongly $q$-concave} if  $\Delta$ is strongly $q$-concave  near each point in $\partial \Delta$.
\end{defn}

\parag{Remark} 
Assume that $Z$ is of pure dimension $q + p$. If  the defining function $\varphi$ of $\Delta$ satisfies Conditions 1 and 2 above, we can compose $\varphi$  with a real strictly increasing (non critical)  convex $\mathscr{C}^{2}$ function (this does not change the level sets  $\{\varphi = \mathrm{constant}\}$), in order that $c\circ\varphi$ is $\mathscr{C}^{2}$ strongly $q$-convex (and non critical) near $z$.

Conversely, if  $\varphi$ is a real valued  $\mathscr{C}^{2}$ function which is  strongly $q$-convex and not critical near at a point $z \in Z$, the open set  $\{ \varphi(x) > \varphi(z)\}$ has a strongly $q$-concave boundary in a neighbourhood of $z$.

\bigskip

With this terminology, using the remarks above, we may give the following reformulation of  Proposition \ref{Hart.}:

  \begin{cor}\label{Hart.cor.}
  Let $n \geq 2$  and  $p \geq 1$  be integers, let  $Z$  be a reduced complex space of pure dimension $ n+p$  and let  $\Delta  := \{\varphi > 0\}$ be an open set with $\mathscr{C}^2$  smooth boundary in $Z$. Let $X_0$ be a $n$-cycle in an open neighbourhood of a point $z \in \partial \Delta$ such that the function $\varphi_{\vert X_{0}}$ is not critical at $z$.

  Assume that  $\Delta$ is strongly $(n-2)$-concave near $z$ ; then there exists a $n$-Hartogs figure $\mathcal{H} := (\mathcal{M}, M, B, j)$ relative to the boundary of $\Delta$, adapted to $X_0$, and such that the point $z$ lies in  $\mathcal{M}$.
 \end{cor}

\begin{proof}
Using a local embedding of an open neighbourhood of $z$ in an open set of the Zariski tangent space $T_{Z,z}$, it is enough to prove the corollary in the case where $Z$ is an open set in $\C^{n+p'}$,  with $p' \geq p$ an integer. As we may choose the function $\varphi$ strongly $(n-2)$-convex such that $d\varphi_{z} \not= 0$ thanks to the previous remarks, we can choose local coordinates near $z$ in order to be in the situation of Proposition \ref{Hart.} in the case $z \in  X_{0}$, as we assumed that $\varphi_{\vert X_{0}}$ is not critical at $z$. In this case the proposition gives the result.

If $z$ is not in $X_{0}$, the same construction in an open neighbourhood of $z$ with no limit point in $X_{0}$ allows to conclude, and in this case the degree of $X_{0}$ in the (adapted) scale  $E_{\mathcal{H}}$  will be zero.
\hfill $\Box$
\end{proof}

  \subsection{Convexity--concavity}
In this paragraph we want to have a brief discussion about $q$-convexity and $q$-concavity.

\medskip

 Let us consider in an open set $U$ of $\C^{n+p}$ a $\mathscr{C}^{2}$ function $\varphi : U \to \mathbb{R}$ and a non critical zero $z_{0}$ on $\varphi$. So $\varphi(z_{0}) = 0$ and $d\varphi_{z_{0}}\not= 0$. Let $D := \{z \in U \ / \  \varphi(z) < 0 \}$ and let $H$ be the complex hyperplane tangent at $z_{0}$ to the real  hypersurface $\{\varphi = 0\}$  which is smooth near $z_{0}$.

 Our terminology (Norguet--Siu convention, see \cite{NS}) is to say that the open set $D$ is strongly $q$-convex near $z_{0} \in \partial D$ if the restriction to $H$ of the Levi form of $\varphi$ at $z_{0}$ has at most $q$ non positive eigenvalues.

 Looking now at the same open set $D$ but asking for some strong concavity condition, we write $D := \{ z \in U \ / \  -\varphi(z) > 0 \}$. Then we shall say that $D$ is strongly $q$-concave at the point $z_{0}$ if the restriction to $H$ of  Levi form of $-\varphi$ at $z_{0}$ has at most $q$ non positive eigenvalues. 

 If the signature of the restriction to $H$ of the Levi form of  $\varphi$ at $z_{0}$ is given by $(p-1)$ ``plus'' and $n$ ``minus'' we see that that near $z_{0}$ our open set $D$ will be strongly $n$-convex near $z_{0}$ and strongly $(p-1)$-concave near $z_{0}$. So $D$ will be strongly $(p-1)$-concave near $z_{0}$  if the  function $-\varphi$ is strongly $(p-1)$-convex at the point $z_{0}$.

 In order that a $\mathscr{C}^{2}$ exhaustion function $\varphi : Z \to ]0, 2]$ on a reduced complex space $Z$ gives relatively compact $q$-concave subsets $Z_{\alpha} := \{\varphi(z) > \alpha \}$ for each $\alpha \in ]0, 1[$ which is not critical for $\varphi$, we see that it is enough that the Levi form of $\varphi$ at each point in $\varphi^{-1}(]0, 1[)$ has at most $q$ non positive eigenvalues. That is to say that $\varphi$ is strongly $q$-convex on this open set.

 In order to reach the key situation given in Proposition \ref{Hart.} with a $n$-cycle, we need to dispose of a $\mathscr{C}^{2}$-exhaustion $\varphi : Z \to ]0, 2]$  which is  $(n-2)$-strongly convex on the open set $\varphi^{-1}(]0, 1[)$ . So we need to assume that $n \geq 2$.

\subsection{Boxed Hartogs figures}

\begin{defn}\label{emboitee}
\rm  Let $n \geq 2$ and $p \geq 1$ be integers and let  $\Delta \subset\subset \Delta'$ be two open sets in a reduced complex space $Z$. Let  $\mathcal{H} = (\mathcal{M}, M, B, j)$ and  $\mathcal{H}' := (\mathcal{H}', M', B, j)$ be two $n$-Hartogs figures in $Z$ relative to the boundary of $\Delta$ given by the same (local) embedding $j$ and having the same polydisc  $B \subset\subset \C^{p}$. We shall say that these two $n$-Hartogs figures are {\bf boxed} when we can choose  $\alpha, \alpha', V, V'$ in Definition \ref{$n-$Marmite} in order to  have
   \begin{itemize}
   \item $\mathcal{M}'(\alpha')\subset\subset \mathcal{M}(\alpha)$,
   \item $M'(\alpha') \subset\subset M(\alpha)$,
   \item    $V' \subset\subset V$.
   \end{itemize}
 \end{defn}

For instance, if   $\varepsilon > 0$ is small enough, the $n$-Hartogs figures  $(\mathcal{H}, \mathcal{H}^{\varepsilon})$ are boxed (see Definition~\ref{homotheties}).

\begin{prop}\label{Rec.M.H.}
  Let $n \geq 2$ and $p \geq 1$ be integers, let $Z$ be a reduced complex space of pure dimension $ n+p$ and let $\Delta \subset \subset Z $ be an open set with smooth $\mathscr{C}^2$ boundary in $Z$ which is strongly  $(n-2)$-concave. Assume that  $\Delta := \{\varphi > 0 \}$ and let  $X_0$  be a $n$-cycle in an open neighbourhood  $\Delta'$  of the compact set $\bar{\Delta}$, such that any irreducible component of $X_{0}$  meets $\Delta$.  Then there exists a finite family of boxed  $n$-Hartogs figures $(\mathcal{H}'_{a}, \mathcal{H}_{a})_{a \in A}$ relative to the boundary of  $\Delta$, such that the following conditions hold:
  \begin{enumerate}
 \item[\rm 1.] The open sets  $\mathcal{M}'_a $ for $a \in A$ cover the boundary $\partial \Delta$.
 \item[\rm 2.] For each $a \in A$ the Hartogs figures $\mathcal{H}_a$ and $\mathcal{H}'_a$ are adapted to $X_0$.
 \item[\rm 3.] For each $a \in A$ any irreducible component of $ X_0$ meeting $\mathcal{M}_{a}$ meets the open set $M'_a$.
 \item[\rm 4.] No {\bf compact} irreducible component of $X_0 \cap \Delta$ meets the union of the compact sets $\bar{\mathcal{M}}_a$, $a \in A$.
 \end{enumerate}
 \end{prop}

 \parag{Remark} Let  $X_0$ be any $n$-cycle in $\Delta'$. Choosing the open set $\Delta'$ small enough around the compact set $\bar{\Delta}$, we can assume that the cycle $X_0$ has only finitely many irreducible components in $\Delta'$ and that each of them which is not compact  meets  $\partial \Delta$ (see Remark 3 following Proposition \ref{Hart.}).

 \begin{cor}\label{Rec.prolgt.}
 In the situation of Proposition \ref{Rec.M.H.}, if we assume that the open set $\Delta'$ containing $\bar{\Delta}$ is small enough, any irreducible component $\Gamma$ of the cycle $X_0$ in $\Delta'$ satisfies for all $a \in A$ and all $\eta > 0$ small enough:
  $$ \Gamma \cap (\mathcal{M}_{a}^{\eta}\times V^{\eta}_{a} \times B_a )  = prlgt_a\big[\Gamma \cap (M_a\times V_a \times B_a) \big]  $$
  where  $prlgt_a : H(\bar M_a\times \bar V_a, \Sym^k(B_a)) \to H(\overline{\mathcal{M}}^{\eta}_{a}\times \bar V^{\eta}_{a}, \Sym^k(B_a))$ is the holomorphic map of analytic extension built in Proposition  \ref{prlgt isotrope}.
  \end{cor}

  \parag{Remark} In the situation of the previous corollary, choosing $\varepsilon > 0$ small enough, there exists, for each $a \in A$, a holomorphic extension map which lifts the map  $prlgt_{a}$:
    $$ iprlgt_{a} : \Sigma_{M_{a}, M^{\varepsilon}_{a}}(k) \longrightarrow \Sigma_{\mathcal{M}_{a}^{\eta}, \mathcal{M}^{\eta+\varepsilon}_{a}}(k) .$$
     It allows to extend in this setting an analytic family of branched coverings in $M_{a}$ parametrized by a  Banach analytic set  $S$ and which is isotropic on $S \times M^{\varepsilon}_{a}$ to an analytic family of branched coverings in $\mathcal{M}^{\eta}_{a}$ which is isotropic on $S \times {\mathcal{M}}^{\eta+\varepsilon}_{a}$.

\begin{proof}[Proof of Proposition  \ref{Rec.M.H.}]
Corollary  \ref{Hart.cor.} and the remark following it implies the existence, for each $z \in \partial\Delta$ of a $n$-Hartogs figure $\mathcal{H}_z$ relative to the boundary of $\Delta$, contained in $\Delta'$ and satisfying the following properties:
      \begin{enumerate}
  \item[i)]  $ z \in \mathcal{M}_z $;
  \item[ii)] $\mathcal{H}_z$  is adapted to $X_0$ ;
  \item[iii)] each irreducible component of  $X_{0}$ meeting $\mathcal{M}_{z}$ meets the open set $M_z $;
  \item[iv)] No compact irreducible component of $X_0 \cap \Delta$ meets $\bar{\mathcal{M}}_z$.
  \end{enumerate}
  As the open sets  $ \mathcal{M}_z$ cover the compact set $\partial \Delta$  we can find a finite sub-covering. Then the properties 1, 2, 3 and 4 are consequences of i), ii), iii) and  iv) by letting
  $\mathcal{H}'_{a} := \mathcal{H}_{a}^{\varepsilon}$  and choosing $\varepsilon > 0$ small enough.
\hfill $\Box$
\end{proof}

\begin{proof}[Proof of Corollary \ref{Rec.prolgt.}] 
Let $\Gamma$ be an irreducible component of $X_0$ meeting $\mathcal{M}_a$ for some $a \in A$.  Then $\Gamma$ meets $M_a$. As $\Gamma$ does not meet $\bar M(\alpha)_{a} \times \bar{V_a}\times \partial B_a $ because $\mathcal{H}_a$ is adapted to $X_0$, the intersection $\Gamma \cap \bar{M}_a$ is the graph of an element $\gamma \in H(\bar{M}(\alpha)_a\times \bar{V_a}, \Sym^{k_{a}}(B_a))$ with $k_{a} \in \mathbb{N}^*$ \footnote{as it exists some $(m,v) \in \bar{M}(\alpha)_a\times \bar{V_a}$ such that $\Gamma \cap (\{m,v)\}\times B_a) \not= \emptyset$.}.

The closed analytic subset $Y$ of the open set $\mathcal{M}^{\eta}_a$ defined by $Y : = prlgt_a\big[\Gamma \cap M_a\big] $
is not empty,  of pure dimension $n$ and is contained in $\Gamma $. So it is a union of irreducible components of $\Gamma \cap \mathcal{M}^{\eta}_a$. But it contains a non empty open set in each irreducible component of this branched covering. So these two analytic subsets coincide.

  If an irreducible component of $X_0$ does not meet any $\bar{\mathcal{M}}_a$ it has to be compact and contained in $\Delta$. In this case the desired equality is obvious.
\hfill $\Box$
\end{proof}

\section{The extension and finiteness theorem}

\subsection{Some useful lemmas}

The version below of Sard's lemma is more or less classical.

  \begin{lemma}\label{Sard}
Let $Z$ be a reduced complex space and let $\varphi : Z \to \mathbb{R}$ be a real valued $\mathscr{C}^{1}$ function. Then the set of critical values of  $\varphi$ has Lebesgue measure $0$.
\end{lemma}

\begin{proof} 
Firstly note that a point $z \in Z$ is critical for  $\varphi$ if, by definition, the differential of $\varphi$ vanishes on $T_{Z,z}$, the Zariski tangent space of $Z$ at $z$. Remember also that a complex space is, by definition, countable at infinity ; so $Z$ and  its singular locus have only countably many irreducible components. As a countable union of sets of measure $0$ is again of measure $0$, it is enough to prove the lemma when $Z$ is irreducible. We shall prove the lemma by induction on the integer $\dim Z$. The case $\dim Z = 0$ is obvious. Assume the lemma true for $\dim Z \leq n-1$ for some integer $n \geq 1$ and take an irreducible complex space $Z$ of dimension $n$. The singular set $S$ of $Z$ has dimension at most $(n-1)$, and for each irreducible component $S_{i}$ of $S$ the image of the critical set of $\varphi_{\vert S_{i}}$ has measure  $0$. So the critical set of  $\varphi_{\vert S}$ is again of measure $0$. But a critical point of  $\varphi$ which belongs to $S$ is a critical point of $\varphi_{\vert S}$. So it is enough to show that the set of critical values of $\varphi$ restricted to the complex connected manifold $Z \setminus S$ has measure $0$. This is the classical Sard's lemma.
\hfill $\Box$
\end{proof}

 \begin{lemma}\label{Ouverture}
Let $V $ be an open set and $K$ be a compact set in $\bar U\times \bar B$. The subset  $\mathcal{V}$ in $H(\bar{U}, \Sym^k(B))$ consisting of the $X$ such that any irreducible component meeting $K$ meets $V$ is an open set in  $H(\bar{U}, \Sym^k(B))$.
\end{lemma}

\begin{proof} 
Let us clarify the meaning of an irreducible component of  an element $X$ in $ H(\bar U, \Sym^{k}(B))$: we call irreducible component of such a $X$ the closure in $\bar U\times B$ of an irreducible component of  the branched covering of $U$ defined by the projection of $X \cap (U \times B)$ on $U$.

Let $X_{0}$ be such that any irreducible component of  $X_{0}$ which meets $K$ meets $V$, and assume that $(X_{\nu})_{\nu \geq 1}$ is a  sequence converging to $X_{0}$ such that for each $\nu \geq 1$ there exists an irreducible component  $\Gamma_{\nu}$ of $X_{\nu}$ meeting $K$ but not $V$. Passing to a subsequence, we may assume that the sequence $(\Gamma_{\nu})_{\nu \geq 1}$ converges uniformely on any compact of $ U\times B$  to a non empty  $n$-cycle $\Gamma$ with closure contained in $X_{0}$ and which is a branched covering of $U$. Then $\bar \Gamma$ meets $K$ and not $V$. Indeed, if $(t_{0}, x_{0})$ would be in $\bar \Gamma \cap V$, there exists open neighbourhoods   $U_{1}$ and $B_{1}$  of $t_{0}$ and  $x_{0}$ respectively in $\bar U$ and   $\bar B$ such that $U_{1}\times B_{1} $ is contained in $V$. But then, as $U_{2} := U_{1}\cap U$ and  $B_{2} := B_{1}\cap B$ are non empty  open sets, for $t_{2}$ in  $U_{2}$ the fibers of the $\Gamma_{\nu}$ at $t_{2}$ for $\nu$ big enough will meet  $\{t_{2}\}\times B_{2}$ and so $V$. As at least one irreducible component of $ \Gamma$ meets $K$ without meeting $V$ and as its closure is an irreducible component of $X_{0}$, this gives a contradiction.
\hfill $\Box$
\end{proof}

Of course, in the case $V = \emptyset$, we get back the fact that the subset in  $H(\bar{U}, \Sym^k(B))$ of  elements  which do not meet $K$ is  open.

\begin{lemma}\label{Recol.}
Let $Z$ be a complex space and let $(\mathcal{U}_i)_{i\in I}$ be an open covering of $Z$. Assume that for each $i \in I$ a closed  $n$-cycle $X_i$ is given in $\mathcal{U}_i$. Assume that  the following patching  condition holds:
$$ \forall (i,j) \in I^2 \quad  X_i \cap \mathcal{U}_j = X_j \cap \mathcal{U}_i $$
as an equality of cycles in  $\mathcal{U}_i \cap \mathcal{U}_j$.
Then there exists a unique closed  $n$-cycle $X$ in $Z$ such that for each $i \in I$ we have $X \cap \mathcal{U}_i = X_i $.
\end{lemma}

For the easy proof see  \cite[Chapter IV, Proposition 1.3.1]{BM1}.

\medskip

The following variant will be used.

\begin{lemma}[Variant]\label{Var.Recol.}
In the situation of the previous lemma replace the patching condition by the following two conditions:
\begin{enumerate}
\item[\rm 1.] For each  couple $(i,j) \in I^2 $ an open subset $W_{i,j} \subset \subset \mathcal{U}_i \cap \mathcal{U}_j $ is given and we ask that $ X_i \cap W_{i,j} = X_j \cap W_{i,j} $.
\item[\rm 2.] For each couple $(i,j) \in I^2 $ we ask that any irreducible component of the cycle $X_i \cap \mathcal{U}_j $ meets the open set $W_{i,j}$.
\end{enumerate}
Then the conclusion is the same.
\end{lemma}

\begin{proof} Let $\Gamma$  be an irreducible component  with multiplicity $\delta$ in the cycle  $X_i \cap \mathcal{U}_j $. Let $\Gamma' $ be the irreducible component of $X_i$ which contains $\Gamma$, and put  $X_i = X'_i + \delta.\Gamma'$. Then $\Gamma'$  meets $W_{i,j}$ and there exists a closed analytic subset  $Y$  of pure dimension $n$ in  $\vert X_j\vert $  such that its restriction to $W_{i,j}$ is equal to  $\Gamma' \cap W_{i,j}$: indeed, $Y$ is the union of the irreducible components of $X_{j}$ containing a non empty open set in  $\Gamma' \cap W_{i,j}$. Note that each of these  irreducible components of $X_{j}$ has multiplicity $\delta$ in the cycle $X_{j}$. Then put   $ X_j = X'_j + \delta.Y$. We see that the cycles  $X'_i$  and   $X'_j$  respectively  in  $\mathcal{U}_i$  and  $\mathcal{U}_j $  satisfy again the patching condition  $X'_i \cap W_{i,j} = X'_j \cap W_{i,j} $.

This allows, for fixed   $(i,j)$, to make a descending induction on the number (necessarily finite as  $W_{i,j}$   is relatively compact) of irreducible components of   $X_i \cap \mathcal{U}_j$, to show that the condition $X_i \cap \mathcal{U}_j = X_j \cap \mathcal{U}_i $ holds. This reduces this lemma to the previous one.
\hfill $\Box$
\end{proof}

\subsection{Adjusted scales.}

The definition  of a scale adapted to a cycle is recalled in Definition \ref{adapted scale}.

 \begin{defn}\label{Ecailles doubles} \rm
  \begin{enumerate}
  \item Let $Z$ be a complex space. We shall call {\bf adjusted $n$-scale on $Z$}, written down  $\mathbb{E} : = (U,U',U'',B, B'', j)$, the data of a $n$-scale on $Z$, $E : = (U,B, j)$, with additional polydiscs $U'' \subset\subset U' \subset \subset U $  and  $B'' \subset\subset B$. We call $E$ the {\bf underlying scale of the adjusted scale  $\mathbb{E}$}.
  \item We shall say that {\bf the adjusted scale $\mathbb{E}$ is adapted to a $n$-cycle $X$} in $Z$ when we have
   $$j^{-1}(\bar U \times (\bar B \setminus B''))\cap \vert X\vert = \emptyset.$$
   Note that this implies that the underlying  scale $E$ is adapted to $X$, but this condition is more restrictive.
   \item When the adjusted scale $\mathbb{E}$ is adapted to the $n$-cycle $X$, we shall call {\bf degree of $X$ in $\mathbb{E}$} the degree of $X$ in $E$.
   \item We shall call {\bf center of the adjusted scale}, written down  $D(\mathbb{E})$, or more simply, $D(E)$, the open set  $j^{-1}(U \times B)$ in $Z$ which is also the center of the scale  $E$.
   \item We shall call {\bf domain of isotropy of the adjusted scale}, written down  $D'(\mathbb{E})$, the open set $j^{-1}(U' \times B)$ in $Z$.
   \item We shall call {\bf domain of patching of the adjusted scale}, written down  $D''(\mathbb{E})$, the open set  $j^{-1}(U'' \times B'')$ in $Z$.
  \end{enumerate}
 \end{defn}

\parag{Remarks} \begin{enumerate}
 \item The open set   $D''(\mathbb{E})$ is relatively compact in   $D'(\mathbb{E})$.
 \item When a  $n$-scale $E$  is given, for any compact set $K$ in  $D(E)$, there exists an adjusted $n$-scale  $\E$ such that $E$ is the underlying scale of $\E$ and with $K \subset D''(\E)$. Moreover, if $E$ is adapted to a $n$-cycle $X_{0}$ in $Z$, we may choose $\mathbb{E}$ in order that it is adapted to $X_{0}$.
 \item As for $X \in \mathcal{C}_{n}^{loc}(Z)$  the condition to avoid a given compact subset is open in $\mathcal{C}_{n}^{loc}(Z)$, when the adjusted scale $\mathbb{E}$ is adapted to a cycle $X_{0}$ there exists an open neighbourhood, written down $\Omega_{k}(\mathbb{E})$, of $X_{0}$ in  $\mathcal{C}_{n}^{loc}(Z)$  such that $\Omega_{k}(\mathbb{E})$ is the subset of all $n$-cycles $X$  in $Z$ for which $\mathbb{E}$ is adapted and $\deg_{\mathbb{E}}(X) = k$ where $k :=  \deg_{\mathbb{E}}(X_{0}) $.
 \end{enumerate}

 \bigskip

 Let $Z$ be a reduced complex space and let $\mathbb{E} = (U,U',U'',B, B'', j)$ be an adjusted scale on $Z$. For a given integer $k$ consider the continuous map sending a branched covering in $H(\bar{U}, \Sym^k(B''))$ to its isotropy data on $\bar U'$ (for the notations see what follows Lemma \ref{banach})
  $$ T : H(\bar{U}, \Sym^k(B'')) \to H(\bar{U}', F \otimes E') .$$
  The graph  $\Sigma_{U,U'}(k)$  of this map is a Banach analytic set\footnote{homeomorphic to $H(\bar{U}, \Sym^k(B''))$  via the projection!},  thanks to  \cite[Propositon 2, p.\,81]{B75}  (see also  \cite{BM2}).

  The set of couples $(f, T(f))$ in $\Sigma_{U,U'}(k)$ for which the associated branched covering is contained in $j(Z\cap \bar D(\mathbb{E}))$ is a closed Banach analytic subset of $\Sigma_{U,U'}(k)$ being the pull-back by the projection of the subset of elements in  $H(\bar{U}, \Sym^k(B''))$ contained in $j(Z\cap \bar D(\mathbb{E})) $ which is a closed Banach analytic subset of  $H(\bar{U}, \Sym^k(B''))$ by Proposition 4, p.\,27 of \cite{B75} (see also \cite[Chapter V]{BM2}).

 \begin{defn}\label{Class. Ec. double} \rm
We shall denote  $\mathcal{G}_k(\mathbb{E})$  this Banach analytic set and we shall call it the {\bf the $k$-th classifying space of the adjusted $n$-scale $\mathbb{E}$} on $Z$.
  \end{defn}

We have then a tautological family of $n$-cycles in the open set $D(\mathbb{E})$ parametrized by $\mathcal{G}_k(\mathbb{E})$. It is an analytic family of cycles in the open set $D'(\E)$, in the sense of \cite{B75}, and the fact that, for $k \geq 1$,  locally on $\mathcal{G}_k(\mathbb{E})$, any irreducible component of a branched covering in this family meets $\bar{U}''\times \bar{B}''$ implies that we have a \textsf{f}-analytic family of cycles in  $D'(\E)$.

\medskip

Be careful that the tautological family of cycles on the open set $D(\mathbb{E})$ parametrized by $\mathcal{G}_k(\mathbb{E})$ is not, in general, an analytic family of $n$-cycles ; see the example of  \cite[p.\,83]{B75} (and also  \cite[Chapter IV]{BM1}).

\medskip

The next lemma is an obvious consequence of {\it loc. cit.}

 \begin{lemma}\label{Quasi-univ.1}
 Let  $\mathbb{E}$ be an adjusted $n$-scale on a reduced complex space $Z$ and let $k$ be an integer. The tautological family of $n$-cycles in the open set $D'(\mathbb{E})$ parametrized by $\mathcal{G}_k(\mathbb{E})$  has the following ``almost universal'' property:
  \begin{itemize}
  \item[]  For any analytic family of $n$-cycles $(X_s)_{s \in S}$ in $Z$ parametrized by a Banach analytic set $S$ such that for each $s \in S$ the adjusted scale $\mathbb{E}$  is adapted to $X_{s}$ with  $\deg_E(X_s) = k$, there exists an unique holomorphic map
  $$ f : S \to \mathcal{G}_k(\mathbb{E})$$
  such that the pull-back by $f$ of the tautological family is the restriction to the open set  $D'(\mathbb{E})$ of the given family.
  \end{itemize}
  \end{lemma}

Of course, conversely, such a holomorphic map gives a \textsf{f}-analytic family of $n$-cycles on the open set $D'(\mathbb{E})$.

Note that the pull-back family is in fact defined on the open set $D(\E)$ but, as already noticed above, it may not be analytic outside  $D'(\mathbb{E})$.

\medskip

  As a consequence of this  ``almost universal'' property, we obtain that for any analytic family $(X_s)_{s \in S}$ of $n$-cycles in $Z$ such that for a point $s_0 \in S$ the adjusted $n$-scale $\mathbb{E}$ is adapted to the cycle $X_{s_{0}}$ with $\deg_{\mathbb{E}}(X_{s_0}) = k$, there exists an open neighbourhood  $S'$  of  $s_0$  in  $S$  such that the previous lemma applies for the family parametrized by $S'$. So we shall have a holomorphic classifying map  $ f : S' \to \mathcal{G}_k(\mathbb{E})$  in this situation.

\medskip

  We shall generalize now the concept of classifying space to the case of  a finite family of adjusted $n$-scales.

  \begin{defn}\label{Don. de Rec.} \rm
Consider a reduced  complex space $Z$ and a finite family of adjusted $n$-scales   $(\E_i)_{i \in I}$   on $Z$. Assume that they are adapted to a given finite type  $n$-cycle  $\hat{X}_0$  in  $Z$. Assume that any irreducible component of $\hat{X}_0$ meets the open set  $W'' : = \bigcup_{i \in I} D''(\E_i)$.

We shall call {\bf patching data for  $\hat{X}_0$  associated to the family  $(\E_i)_{i \in I}$}, written $\mathcal{R}((\E_i)_{i \in I}, F)$  or more simply  $\mathcal{R}$  when there is no ambiguity, a finite collection of $n$-scales $(F_{i,j,h})$ for $(i,j) \in I^2$, $i \not= j$,  where  $h$  belongs to a finite set  $H(i,j)$ for each couple $(i,j) \in I^2$, $i \not= j $, such that the following properties hold:
\begin{enumerate}
\item[i)] $F_{i,j,h} $ is a  $n$-scale on the open set $D'(\mathbb{E}_{i}) \cap D'(\mathbb{E}_{j})$.
\item[ii)] The $n$-scales  $F_{i,j,h} $ are adapted to $\hat{X}_{0}$.
\end{enumerate}
We shall say that the patching data  $\mathcal{R}$ are {\bf complete} when  the following condition also holds:
\begin{enumerate}
\item[iii)] For each  $i \not= j$ given, the union of domains of the scales  $F_{i,j,h}$, $h \in H(i,j)$, covers the compact subset $\overline{D''(\mathbb{E}_{i})} \cap \overline{D''(\mathbb{E}_{j})}$ of $D'(\mathbb{E}_{i})\cap D'(\mathbb{E}_{j})$.
\end{enumerate}
\end{defn}

 \parag{Notations}\begin{enumerate}
\item In the sequel, when we shall consider a reduced complex space $Z$ and a finite family of adjusted $n$-scales $(\E_i)_{i \in I}$, adapted to a finite type $n$-cycle $\hat{X}_0$ in $Z$, such that any irreducible component of  $\hat{X}_0$ meets the open set
 $$W'' : = \bigcup_{i \in I} D''(\E_i),$$
 we shall say that the family $(\E_i)_{i \in I}$ is {\bf convenient} for  $\hat{X}_0$.
\item In this setting we shall use the following definitions :
\begin{itemize}
\item $W : = \bigcup_{i \in I} D(\E_i)$ ;
\item $W' : = \bigcup_{i \in I} D'(\E_i)$ ;
\item $W'' : = \bigcup_{i \in I} D''(\E_i)$ ;
\item $K : =  \bigcup_{i\in I}  j_i^{-1}(\bar{U}_i \times (\bar B_i\setminus B''_{i}))$.
\item When the family $(\E_i)_{i \in I}$ is convenient for a finite type $n$-cycle $\hat{X}_0$, $\mathcal{K}$ will be a compact neighbourhood of $K$ disjoint from $\hat{X}_0$.
\end{itemize}
\item For $\tilde{X} \in \prod_{i \in I} \mathcal{G}_{k_i}(\E_i)$ we shall denote by $X^i$ the closed cycle in $D(\E_i)$ associated to the $i$-th component of $\tilde{X}$.
\end{enumerate}

\begin{lemma}\label{tech.}
Let  $(\E_i)_{i \in I}$  be a finite family of adjusted $n$-scales on $Z$, convenient for a $n$-cycle $\hat{X}_0$ of finite type in $Z$, and let $\mathcal{R}$ be some corresponding  complete patching data. There exists an open neighbourhood $\mathcal{V}$ of the image $\tilde{X}_0$ of  $\hat{X}_0$ in the product  $\prod_{i \in I} \mathcal{G}_{k_i}(\E_i)$ such that  for each $\tilde{X} \in \mathcal{V}$  we have the following properties:
\begin{enumerate}
\item[\rm 1.] No $X^i$ meets the compact set $\mathcal{K}$.
\item[\rm 2.] For each $i \in I$, any irreducible component of  $X^i$ meeting  $\overline{D''(\mathbb{E}_{i})} \cap \overline{D''(\mathbb{E}_{j})}$ with $j \not= i$, meets the open set $\bigcup_{h} D(F_{i,j,h})$.
\item[\rm 3.]  For each $(i,j,h)$  the scale $F_{i,j,h}$  is adapted to  $X^i$   and   $X^j$.
\item[\rm 4.]  For each $(i,j,h)$ we have $\deg_{F_{i,j,h}}(X^i) = \deg_{F_{i,j,h}}(X^j) = \deg_{F_{i,j,h}}(\hat{X}_0) = k_{i,j,h}$.
\end{enumerate}
\end{lemma}

\begin{proof} Conditions 1, 3 and 4  are clearly open. An easy consequence of Condition 1, of the inclusion ii) of Definition \ref{Don. de Rec.} and of Lemma \ref{Ouverture} is that Condition 2 is also open.
\hfill $\Box$
\end{proof}

For a $n$-scale  $E := (U, B, j)$ on $Z$ we shall abbreviate $H(\bar U, \Sym^{k}(B))$ in $G_{k}(E)$.

\medskip

When we consider a cycle $X_{0}$ in an adapted scale $E:=(U, B, j)$ and when we dispose of a $n$-scale $F := (V, C, h)$ on $U\times B$, adapted to $X_{0}$, where $h$ is given by an isomorphism of an open set in $U\times B$ into some open neighbourhood of $\bar V\times \bar C$ in $\C^{n}\times \C^{p}$, we have a well-defined map of an open neighbourhood $\mathcal{U}$ of $X_{0}$ is $H(\bar U, \Sym^{k}(B))$ into $H(\bar V, \Sym^{l}(C))$, where $l := \deg_{F}(X_{0})$, sending $X \in \mathcal{U}$ to the multiform graph associated to $h_{*}(X)$ in the scale $F$. This is a consequence of the fact that the condition $X_{0}\cap h^{-1}(\bar V\times \partial C) = \emptyset$  is open in $H(\bar U, \Sym^{k}(B))$ and that the degree of $X$ near enough $X_{0}$ in the adapted scale $F$ will be equal to $l$.

Such a map, which will be called a {\bf change of scale}, is not holomorphic in general but becomes holomorphic when we add the isotropy condition:

precisely, if $U' \subset\subset U$ and if  $h^{-1}(\bar V\times \bar C) \subset U'\times B$, then the change of scale map
$$ \Sigma_{U,U'} \to H(\bar V, \Sym^{l}(C))$$
will be holomorphic (see Theorem 4, p.\,66 in \cite{B75}).

\begin{defn}\label{Class. fam. conv.} \rm
Let   $(\E_i)_{i \in I}$   a finite family of adjusted $n$-scales on $Z$, convenient for a $n$-cycle  $\hat{X}_0$ of finite type in $Z$, and let  $\mathcal{R}$  be some corresponding  patching data.
Let  $k_i$  be the degree of  $\hat{X}_0$   in the adjusted scale  $\E_i$. For each  $(i,j,h)$, $i \not= j$  we have a couple of holomorphic maps
$$\xymatrix{ \prod_{\alpha \in I} \mathcal{G}_{k_{\alpha}}(\mathbb{E}_{\alpha}) \ar@<1 ex> [r] \ar@<-1ex> [r] & G_{k_{i,j,h}}(F_{i,j,h})}$$
obtained by the changes of scales   $\E_i \to F_{i,j,h}$   and    $\E_j \to F_{i,j,h}$, because, by construction, we have  $D(F_{i,j,h}) \subset \subset D'(\E_i) \cap D'(\E_j)$.

We shall denote by $S(\mathcal{R})$  the intersection of the kernels of these double maps\footnote{The kernel of a double map $f,g : A \to B$ is the pull-back of the diagonal in $B \times B$ by the map $(f,g) : A \to B\times B$.} with the open set  $\mathcal{V}$  built in Lemma \ref{tech.}.  It is a Banach analytic set and we shall call it the {\bf classifying space associated to  $(\E_i)_{i \in I}, \hat{X}_0$ and $ \mathcal{R}$}.
  \end{defn}

  Remark that the patching data $\mathcal{R}$ are not assumed to be complete in the previous definition.

  \begin{prop}\label{Pte q-univ. class.}
  Consider a finite family of adjusted $n$-scales which is convenient for the $n$-cycle  $X_0$   in $Z$ and let  $\mathcal{R}$   be some  {\bf complete} patching data associated. Keeping the previous notations we have for each   $(X^i)_{i \in I} \in S(\mathcal{R})$  a unique $n$-cycle $X \in \mathcal{C}^{\mathsf{f}}_n(W')$ such that  $X \cap D'(\E_i) = X^i \cap D'(\E_i)$, $\forall i \in I $.

  Moreover, this defines a  tautological family of cycles in  $W'$  which is a $\mathsf{f}$-analytic family of cycles satisfying the following ``almost universal'' property:
  \begin{itemize}
 \item[] For any analytic family of $n$-cycles  $(X_s)_{s \in S}$  in an open neighbourhood of  $\bar{W}$  parametrized by a Banach analytic set  $S$\footnote{Recall that, by definition, this means that for any $n$-scale $E := (U, B, j)$ on $Z$,  adapted to some $X_{s_{0}}, s_{0} \in S$, with $\deg_{E}(X_{s_{0}}) = k$, we have an open neighbourhood $S_{0}$ of $s_{0}$ in $S$ and a classifying map for the corresponding family of branched coverings $f : S_{0}\times U \to \Sym^{k}(B)$ which is {\bf isotropic}.}  such that for  $s_0 \in S$  we have  $X_{s_0} = X_0$  in a neighbourhood of   $\bar{W}$, there exists an open neighbourhood  $S'$  of  $s_0$   in   $S$   and a unique holomorphic map
  $$ f : S' \to S(\mathcal{R})$$
 such that the pull-back by   $f$  of the tautological family parametrized by   $ S(\mathcal{R})$  is the restriction to the open set   $W'$    of the family  $(X_s)_{s \in S'}$.
  \end{itemize}
  \end{prop}

\begin{proof} For each  $\tilde{X} \in \mathcal{V}$   any   $X^i$   does not meet   $\mathcal{K}$. As $\mathcal{R}$ is complete, we may use Lemma \ref{Var.Recol.} with   $\mathcal{U}_i : = D'(\E_i)$ and $W_{i,j} = \bigcup_h D(F_{i,j,h})$  to associate to   $\tilde{X}$   a  finite type $n$-cycle $X$   of the open set   $W'$.~The \textsf{f}-analyticity of the so defined family is obvious.
  The ``almost universal'' property is then clear.
\hfill $\Box$
\end{proof}

  Note that this proposition implies that the map $S(\mathcal{R}) \to \mathcal{C}_{n}^{\mathsf{f}}(W')$ classifying the tautological family of $n$-cycles in $W'$ is a holomorphic map.

  \subsection{Shrinkage.}

   \begin{defn}\label{Homoth.} \rm
   Let $Z$ be a reduced complex space and let    $\mathbb{E} : = (U, U', U'', B, B'', j)$  be an adjusted $n$-scale on $Z$. For any real  $\tau > 0$  small enough we shall denote by  $\mathbb{E}^{\tau}$  the adjusted  $n$-scale on $Z$ defined as   $\mathbb{E}^{\tau} : = (U^{\tau}, {U'}^{\tau}, U'', B, B'', j )$. We shall call   $\mathbb{E}^{\tau}$ the $\tau$-shrinkage of  $\E$.
    \end{defn}

     Recall that for a polydisc $P$ of radius $R$, $P^{\tau}$ is the polydisc with same center and radius $R - \tau$. The definitions of  $\mathcal{M}(\alpha)^{\tau}$, $M(\alpha)^{\tau}$ are given in the section \hyperref[2.2]{2.B}.

\medskip

\parag{Remarks}
\begin{enumerate}
\item By definition,  the shrinkage of $\E$ does not change the embedding $j$ and  the polydiscs $U'', B, B''$.
\item As $j$ is a closed embedding of an open set in $Z$ in an open neighbourhood of the compact set  $\bar{U}\times \bar{B}$, it is clear that for any given adjusted  $n$-scale  $\E$  on $Z$, there exists a real  $\varepsilon > 0 $ (depending on  $\E$) such that for any  $\tau \in ]0, \varepsilon[$, $\E^{\tau}$ is again an adjusted $n$-scale on $Z$.
\item If $\E$ is an adjusted scale adapted to the $n$-cycle  $X_0$  in  $Z$, for $\tau$ small enough (depending on  $\E$  and   $X_0$), the adjusted $n$-scale  $\E^{\tau}$  remains adapted to    $X_0$  and we shall have also   $\deg_{E^{\tau}}(X_0) = \deg_E(X_0)$.
\item If $\E$ is an adjusted scale adapted to the $n$-cycle   $X_0$ in $Z$, there exists an open neighbourhood  $\mathcal{U}$  of  $X_0$   in  $\mathcal{C}^{loc}_n(Z) $  and a real $\varepsilon > 0$ such that for any  $X \in \mathcal{U}$  and any  $\tau \in ]0,  \varepsilon[$, the adjusted scale   $\E^{\tau}$  remains adapted to $X$ with again   $\deg_{E^{\tau}}(X) = \deg_E(X_0)$.
\item In the situation of the previous proposition \ref{Pte q-univ. class.}, we may, for   $\tau > 0$  small enough, keep the same  patching data  $\mathcal{R}$   on the  finite family  $(\E^{\tau}_i)_{i \in I}$ of adjusted $n$-scales;  if it was complete,  it remains complete and if it was  convenient\footnote{See the notations following Definition \ref{Don. de Rec.}.} for  the $n$-cycle $X_0$, it remains convenient for  the $n$-cycle  $X_0$. Then we have a holomorphic  restriction map
 $$ S(\mathcal{R}) \to S^{\tau}(\mathcal{R})$$
 where $S^{\tau}(\mathcal{R})$ is the classifying space associated to the family $(\mathbb{E}^{\tau})_{i \in I}$, and this map is induced by a {\bf finite product of linear (continuous) compact maps}. This last point is crucial for the finiteness results.
  \end{enumerate}

\subsection{Excellent family.}
In order to avoid that our notations become  too heavy we shall introduce the following conventions when $\mathcal{H}$ is a $n$-Hartogs figure in $\C^{n+p}$ ; we  shall put, using the notations introduced above
 $$U := M(\alpha)\times V, \quad U' :=  M(\alpha)^{\varepsilon'}\times V^{\varepsilon'}, \quad U'' := M(\alpha)^{\varepsilon''}\times V^{\varepsilon''}, \quad B'' := B^{\varepsilon'}$$
 where $ \varepsilon' > 0$ is small enough, and where $0 < \varepsilon'' < \varepsilon'$. The choices of $\varepsilon'$ and  $\varepsilon''$ will be precised when they are useful. We shall associate to  $\mathcal{H}$ the adjusted  $n$-scale on  $\Delta$ given by:
   $$\mathbb{E}_{\mathcal{H}} := (U, U', U'', B, B'', j).$$
We shall put also
 $$ \tilde{U} := \mathcal{M}(\alpha)\times V, \quad  \tilde{U}' := \mathcal{M}(\alpha)^{\varepsilon'}\times V^{\varepsilon'}, \quad \tilde{U}'' := \mathcal{M}(\alpha)^{\varepsilon''}\times V^{\varepsilon''}$$
 the holomorphy envelopes respectively of $U$, $U'$ and  $U''$. Then we shall have a family of adjusted $n$-scales on  $\Delta'$, written down:
  $$ \mathbb{E}^{\eta}_{\tilde{\mathcal{H}}}:= (\tilde{U}^{\eta}, \tilde{U'}^{\eta}, \tilde{U''}, B, B'', j) \quad {\rm with} \quad \tilde{U}^{\eta} := \mathcal{M}(\alpha)^{\eta}\times V^{\eta}, \ \tilde{U'}^{\eta} := \mathcal{M}(\alpha)^{\varepsilon' +\eta}\times V^{\varepsilon' + \eta},$$
  where $\eta$ is a non negative real number, small enough (for $\eta = 0$ we shall simply write $\mathbb{E}_{\tilde{\mathcal{H}}}$).

  Then we shall have the isotropic classifying spaces
  $$\mathcal{G}_{k}(\mathcal{H}) := \Sigma_{U,U'}(k) \quad { \rm and~also} \quad  \mathcal{G}_{k}^{\eta}(\tilde{\mathcal{H}}) := \Sigma_{\tilde{U}^{\eta}, \tilde{U'}^{\eta}}(k).$$
   Proposition \ref{prlgt isotrope}  gives a holomorphic analytic extension map $prlgt^{\eta}$  such that the following diagram commutes:
   $$\xymatrix{ \mathcal{G}_{k}(\tilde{\mathcal{H}}) \ar[d]_{res} \ar[r]^{res^{\eta}} & \mathcal{G}_{k}^{\eta}(\tilde{\mathcal{H}}) \\ \mathcal{G}_{k}(\mathcal{H}) \ar[ru]_{prlgt^{\eta}}& \quad}$$

   \begin{defn}\label{Excellent.} \rm
   Let $Z$ be a reduced complex space of pure dimension $n+p$, let  $\Delta \subset \subset Z$  be a strongly $(n-2)$-concave open set in $Z$ and   $\tilde{X}_0$    a $n$-cycle in an open neighbourhood    $\Delta'$  of $\bar \Delta$ in  $ Z$. We shall say that a finite family $(\mathcal{H}_a)_{a \in A}$   of $n$-Hartogs figures relative to the boundary of  $\Delta$  is {\bf excellent} for the cycle  $\tilde{X}_{0}$   when the following conditions hold, where we write  $\tilde{\E}_{a}$ and  $\E_{a}$ the adjusted scales respectively on $\Delta'$ and  $\Delta$ associated to the $n$-Hartogs figure $\mathcal{H}_{a}$ :
   \begin{enumerate}
 \item The adjusted scales  $\tilde{\E}_{a}$ and  $\E_{a}$ are adapted to the cycle $\tilde{X}_{0}$.
\item We may choose the patching domains of the adjusted scales   $(\tilde{\E}_{a})_{a \in A}$ in order that the union
 $$\tilde{D}''(A) := \bigcup_{a \in A}   j_{a}^{-1}(\tilde{U''}_{a}\times B''_{a})$$
contains the compact  set  $\partial \Delta$.
 \item There exists a finite family of adjusted $n$-scales   $(\E_{b})_{b \in B}$ on $\Delta$, adapted to  $\tilde{X}_0$, such that the finite families  $(\E_{c})_{c \in A \cup B}$   and   $(\tilde{\E}_{c})_{c \in A \cup B}$   are convenient  for   $\tilde{X}_0$, where we put   $\tilde{\E}_b = \E_b$   for   $b \in B$.

 Moreover we ask that the union    $D''(B)$   of the patching domains of the   $(\E_b)_{b \in B}$   covers the compact set   $\Delta \setminus \tilde{D}''(A) $   of   $\Delta$; so   $\tilde{D}''(A) \cup D''(B)$   in an open set containing   $\bar{\Delta}$.
    \end{enumerate}
  \end{defn}

As a consequence, the union $\tilde{D}'(A)^{\eta}\cup D'(B)$ of the isotropy domains will cover $\bar \Delta$ for $\eta > 0$ small enough.

\begin{prop}[Existence of  excellent families]\label{Exist.Excel.}
Let $Z$ be a reduced complex space of pure dimension $n + p$,   $\Delta \subset \subset Z$   be a strongly $(n-2)$-concave open set with smooth boundary and   $\tilde{X}_0$   a   $n$-cycle in an open neighbourhood   $\Delta' \subset Z$  of  $\bar \Delta$   such that any irreducible component of   $\tilde{X}_0$   meets   $\Delta$.

Then there exists a finite family   $(\mathcal{H}_a)_{a \in A}$   of $n$-Hartogs figures relative to the boundary of   $\Delta$   which is excellent for the cycle   $\tilde{X}_0$.
\end{prop}

\begin{proof} First we use Proposition \ref{Rec.M.H.} to cover  $\partial \Delta$  by a finite family of $n$-Hartogs figures relative to the boundary of   $\Delta$   adapted to the cycle   $\tilde{X}_0$  such that Conditions 1 and 2 hold. Then we build a finite family of adjusted $n$-scales   $(\E_b)_{b \in B}$  on   $\Delta$, adapted to  $\tilde{X}_0$   in order that Condition 3 holds.
\hfill $\Box$
\end{proof}

  \subsection{The extension and finiteness theorem.}

  The next theorem will be crucial in the proof  of Theorem \ref{germe}.

  \begin{thm}\label{Th. de finitude}
Let $Z$ be a reduced complex space of pure dimension $n+p$, where   $n \geq 2$, $p \geq 1$. Assume that there exists a $\mathscr{C}^{2}$ exhaustion $ \varphi : Z \to ]0, 2]$ which is strongly $(n-2)$-convex on the open set $\varphi^{-1}(]0, 1[)$ and let $\Delta := \{ x \in Z \ / \  \varphi(x) > \alpha\}$ for some $\alpha \in ]0, 1[$ which is not a critical value for $\varphi$.  Let   $\tilde{X}_0$    a closed $n$-cycle of an open neighbourhood    $\Delta'$    of   $\bar{\Delta}$    in $Z$ such that any irreducible component of   $\tilde{X}_0$   meets $\Delta$. Then there exists an open neighbourhood $\Delta''$ of $\bar \Delta$  in $\Delta'$  and a $\mathsf{f}$-analytic family  $(\tilde{X}_{\xi})_{\xi \in \Xi}$  of $n$-cycles in  $\Delta''$  parametrized by a reduced complex space  $\Xi $  (of finite dimension)  such that $X_{\xi_{0}} = \tilde{X}_0 \cap \Delta''$ and such that $\Xi$ is isomorphic to an open neighbourhood of  $\tilde{X}_{0} \cap \Delta''$  in $\mathcal{C}_{n}^{\mathsf{f}}(\Delta'')$. It satisfies the following universal property:
\begin{itemize}
\item[] For any  $\mathsf{f}$-analytic  family   $(X_s)_{s \in S}$   of $n$-cycles in   $\Delta$   parametrized by a Banach analytic set $S$ and such that its value at some  $s_{0} \in S$ is equal to   $\tilde{X}_0\cap \Delta $, there exists an open neighbourhood $S'$ of $s_{0}$ in $S$ and an unique holomorphic map
$$ h : S' \to \Xi $$
satisfying  the equality   $X_s  = \tilde{X}_{h(s)}\cap \Delta$    for each   $s \in S'$.
\end{itemize}
\end{thm}

\begin{proof} Note that $\Delta$ is a relatively compact open set in $Z$ with a $\mathscr{C}^{2}$ boundary which is strongly $(n-2)$-concave.

Begin by covering the compact set   $\partial \Delta$   by an excellent finite family $(\mathcal{H}_a)_{a \in A}$   of $n$-Hartogs figures for $\Delta$ adapted to the cycle  $\tilde{X}_0$. Choose then open sets   $\Delta_{1}\subset\subset \Delta \subset\subset \Delta'' \subset\subset \Delta'$, such that the following properties hold, where we use the notations introduced above for the finite family of the adjusted scales  $(\tilde{\E}_a)_{a \in A}$   associated to $(\mathcal{H}_a)_{a \in A}$:
\begin{enumerate}
\item[i)] $\bar{\Delta}'' \setminus \Delta_{1} \subset  \bigcup_{a \in A} W''(\tilde{\E}_a) $.
\item[ii)] $ \bigcup_{a \in A} W(\tilde{\E}_a) \subset\subset \Delta' $.
\item[iii)] $ \bigcup_{a \in A} \bar{W}(\E_a) \subset \Delta $.
\item[iv)] $K : = \bigcup_{a \in A} j_a^{-1}(\bar{U}_a\times (\bar B_a\setminus B''_{a})) \subset \Delta_{1}$
\end{enumerate}
It is easy to fulfill  these conditions for  $\Delta_{1}$   and   $\Delta''$   near enough to   $\Delta$, as, by assumption, the  subsets   $\bar{W}(\E_a)$   and    $j_a^{-1}(\bar{U}_a\times \bar B_a\setminus B''_{a})$   are compact in   $\Delta$, and as the union of the open sets   $W''(\tilde{\E}_a)$  contains   $\partial \Delta$.

Note that Condition iv) allows to choose the compact neighbourhood    $\mathcal{K}$  of   $K$ inside   $\Delta_{1}$.

Choose now a convenient finite family  $(\E_b)_{b\in B}$    of adjusted $n$-scales on $\Delta$, adapted to   $\tilde{X}_0$   in order that the open set  $\bigcup_{b \in B}  W''(\E_b)$  contains  $\bar{\Delta}_{1}$.

Put   $\tilde{\E}_b = \E_b$  for   $b \in B$, and define   $C : = A \cup B$. The family  of adjusted scales    $(\E_c)_{c \in C}$   in  $\Delta$   is then convenient for  $\tilde{X}_0 \cap \Delta_{1}$, up to choosing the patching domains big enough. Fix some complete patching data  $\mathcal{R}$   associated to the family  $(\E_c)_{c \in C}$.

The family  $(\tilde{\E}_c)_{c \in C}$  of adjusted scales in $\Delta'$ is convenient for   $\tilde{X}_0 \cap \Delta'' $ if we choose the patching domains big enough. Consider now some complete patching data for this family of the form  $\mathcal{R} \cup \mathcal{R}''$, that is to say containing the patching scales already in   $\mathcal{R}$. Define the following Banach analytic sets:
\begin{enumerate}
\item    $ S_0 $    is the classifying space of the family   $(\tilde{\E}_c)_{c \in C}$, the degrees being these of  $\tilde{X}_{0}$ in the various scales adapted  to the cycle $\tilde{X}_0$, with the patching conditions defined by   $\mathcal{R}$. Note that  $\mathcal{R}$ is not complete in general.
\item  $S_+$   is the classifying space of the family  $(\tilde{\E}_c)_{c \in C}$  with the patching conditions defined by   $\mathcal{R} \cup \mathcal{R}''$.
\item  $S_{-}$   is the classifying space of the family  $(\E_c)_{c \in C}$, with the (complete)  patching conditions defined by   $\mathcal{R}$.
\end{enumerate}
Then we get a holomorphic extension map
$$ \alpha : S_{-} \to S_0 $$
deduced from the extension maps in the $n$-Hartogs figures $(\mathcal{H}_{a})_{a \in A}$.

By definition   $S_+$   is a closed Banach analytic subset of    $S_0$ as it is defined in $S_0$ by the patching conditions given by   $\mathcal{R}''$. Then put   $\Xi : = \alpha^{-1}(S_+)$. So we have a holomorphic map  $\alpha : \Xi \to S_+$. We want to show the following claim:
\parag{Claim}
There exists a holomorphic map  $\beta : S_+ \to \Xi $   satisfying the  two properties:
\begin{enumerate}
\item  We have   $\alpha \circ \beta = Id$   and   $\beta\circ \alpha = Id$  in a neighbourhood of the points defined by   $\tilde{X}_0$   respectively in  $S_+$   and   $\Xi$.
\item The holomorphic map  $\beta$   is the composition of a holomorphic map induced by a linear (continuous) compact map and a holomorphic map.
\end{enumerate}

\smallskip

As the open set   $ \bigcup_{c \in C} W''(\tilde{\E}_c) $  contains   $\Delta''$, there is, on  $S_+$, a tautological family of $n$-cycles which is \textsf{f}-analytic on  $\Delta''$. Then the ``almost universal'' property of   $S_{-}$  gives a holomorphic map  $\tilde{\beta} : S_+ \to S_{-} $   which factorizes via the closed Banach analytic subset  $\Xi \subset S_{-}$. Let us show that the holomorphic map   $\beta : S_+ \to \Xi $   deduced from this factorization satisfies the two properties of the claim.

First we have    $\alpha\circ\beta = Id$   and   $\beta\circ\alpha = Id$  respectively in  $S_+$   and  $\Xi$, by definition of    $\Xi$ and $\beta$ (see Proposition \ref{prlgt isotrope}).

To see the second property, consider $\tau > 0$ small enough and let us show that the holomorphic map induced by the  linear compact (by Vitali's theorem) restriction   $r^{\tau} : S_+ \to S^{\tau}_+ $   factorizes   $\beta$, where $S_+^{\tau}$   is  the classifying space corresponding to the $\tau$-shrinkage  $(\tilde{\E}^{\tau}_c)_{c \in C}$   of the family   $(\tilde{\E}_c)_{c \in C}$   with the patching data   $\mathcal{R} \cup \mathcal{R}''$. Indeed, for $\tau$ small enough, the tautological family of    $S^{\tau}_+ $    is still \textsf{f}-analytic on   $\Delta$    and the ``almost universal'' property of $S_{-}$   gives again a holomorphic map   $\tilde{\beta}^{\tau} : S^{\tau}_+  \to S_{-}$. And we have    $\beta = \tilde{\beta}^{\tau}\circ r^{\tau}$   proving our assertion. In fact, the map   $\tilde{\beta}^{\tau}$   takes its values in   $\Xi$   because the $\tau$-shrinkage does not change  the patching data deduced from  $\mathcal{R} \cup \mathcal{R}''$ for $\tau$ small enough.

We conclude that the Banach analytic set  $\Xi$  has finite dimension thanks to the finiteness lemma of  \cite[p.\,8]{B75} (see also \cite[Chapter V]{BM2}). Moreover, it is isomorphic to  $S_+$  which is also of finite dimension.

The holomorphic isomorphism  $\beta : S_{+} \to \Xi$ factorizes via an open neighbourhood of $X_{0}\cap \Delta''$ in $ \mathcal{C}_{n}^{\mathsf{f}}(\Delta'')$ because $S_{+} $ parametrizes a \textsf{f}-analytic family of cycles in $\Delta''$, and because a cycle $X$ near enough to $X_{0}\cap \Delta''$ defines an element in $\Xi \subset S_{-}$ as it satisfies the patching conditions $\mathcal{R} \cup \mathcal{R}''$. This shows that we can identify the Banach analytic set $\Xi$ (which is a reduced finite dimensional complex space)  with an open neighbourhood of $X_{0}\cap \Delta''$ in $\mathcal{C}_{n}^{\mathsf{f}}(\Delta'')$.

The universal property is  obvious.
\hfill $\Box$
\end{proof}

Note that it is not restrictive to choose $\Delta'' := \{\varphi > \alpha_{1}\}$ for some $\alpha_{1}\in ]0, \alpha[$, near enough to $\alpha$.

\medskip

\parag{Remarks}\begin{enumerate}
\item  The reduced complex space (of finite dimension)  $\Xi$  built in the previous theorem parametrizes a \textsf{f}-analytic family of $n$-cycles in the open set  $\Delta''$  which is an open neighbourhood of $\bar{\Delta}$. So, when we have a  \textsf{f}-analytic  family $(X_s)_{s \in S}$  of $n$-cycles in $\Delta$  parametrized by a Banach analytic set $S$ and such that its value at some  $s_{0} \in S$ is equal to  $\tilde{X}_0\cap \Delta $ we can extend each cycle $X_s$, $s \in S'$, to a $n$-cycle $\tilde{X}_s$ in $ \Delta''$ in order that the family $(\tilde{X}_s)_{s \in S'}$ is \textsf{f}-analytic in $\Delta''$, with the condition that each irreducible component of $\tilde{X}_s$ meets $\Delta$ and with  the equality $\tilde{X}_s \cap \Delta = X_{s}$ for each $s \in S'$.
\item We shall see later on  that  $\Xi$ is also (isomorphic to) an open neighbourhood of the point  $\tilde{X}_{0}\cap \Delta$  in   $\mathcal{C}_{n}^{\mathsf{f}}(\Delta)$.
\end{enumerate}

  \section{Finiteness of the space of $n$-cycles of a reduced strongly $(n-2)$-concave space ($n \geq 2$).}

  \subsection{The global extension theorem.}

  First we have to complete our terminology.

  \begin{defn}\label{n-concave} \rm
  We shall say that a reduced complex space $Z$ is {\bf strongly $q$-concave}, where $q \geq 0$ is a natural integer, if there exists a real valued $\mathscr{C}^{2}$ exhaustion function on $Z$,
  $\varphi : Z \to  ]0,2]$,  which is strongly $q$-convex outside the compact set $K := \varphi^{-1}([1, 2])$.
   \end{defn}

   In the sequel, when we consider a reduced complex space $Z$ which is assumed to be $q$-concave, we shall always assume implicitly that we have chosen  such an exhaustion $\varphi$. For instance, any reduced compact complex space is strongly $q$-concave for any $q \geq 0$.

   When $Z$ is strongly $q$-concave irreducible and non compact of dimension at least $q +1$, the function $\varphi$ achieves its maximum at a point in $K$. So we shall have    $\varphi(Z) = ]0, u]$ with $u \geq 1$.

   \begin{thm}\label{prlgt global}
Let $n \geq 2$ be an integer. Let $Z$ be a reduced complex space which is strongly $(n-2)$-concave. Let  $\alpha \in ]0, 1[$ and let   $X$  be a finite type  $n$-cycle in the open set  $Z_{\alpha} := \{z \in Z \ / \ \varphi(z) > \alpha \}$. Then there exists a unique $n$-cycle $\tilde{X}$ in  $\mathcal{C}_{n}^{\mathsf{f}}(Z)$ such that  $\tilde{X} \cap Z_{\alpha} = X$.
\end{thm}

The proof will use the following remark.

\parag{Remark} Consider a closed irreducible analytic subset $\Gamma$ of dimension $n$  in $Z$. As the restriction  $\varphi_{\vert \Gamma}$ of the exhaustion $\varphi$ to $\Gamma$ must reach its maximum (as $\varphi$ is continuous and proper), this maximum cannot be obtained at a point in which  $\varphi$ is strongly $(n-2)$-convex. So we have $\Gamma \cap \varphi^{-1}[1,2]) \not= \emptyset$.

\begin{proof} The uniqueness of $\tilde{X}$ is a consequence of the previous remark:
if $\tilde{X}_{i}$, $i = 1, 2$, are in $\mathcal{C}_{n}^{\mathsf{f}}(Z)$ and satisfy $\tilde{X}_{i} \cap Z_{\alpha} = X$, then any irreducible component $\Gamma_{1}$ of $\tilde{X}_{1}$ has to meet $Z_{\alpha}$ and so has to contain an open set in $\tilde{X}_{2}$. Then it has to be an irreducible component of $\tilde{X}_{2}$ and its multiplicity in $\tilde{X}_{1}$ and $\tilde{X}_{2}$ must coincide. So $\tilde{X}_{1} = \tilde{X}_{2}$.

To show that this cycle exists, consider first the case where $X$ is compact in  $Z_{\alpha}$. Then  $\tilde{X} := X$ is a solution. So it enough to consider the case where $X$ is irreducible and non compact. Thanks to \cite[Theorem 8.3]{ST}, for each $z \in \partial Z_{\alpha}$ there exists an open set  $U_{z}$  and an unique closed  analytic set  $X_{z}$ in $Z_{\alpha} \cup U_{z}$ of pure dimension $n$ such that  $X_{z} \cap Z_{\alpha}= X$. Choose a finite set of points  $z_{1}, \dots, z_{N}$ and open sets  $U'_{i}\subset\subset U_{i}:= U_{z_{i}}$ such that the union of the $U'_{i}$ covers the compact set $\partial Z_{\alpha}$. Let  $\Omega := Z_{\alpha}\cup \big(\bigcup_{i=1}^{N} U'_{i}\big)$ and put
 $$X_{1} := (X \cup \big( \bigcup_{i=1}^{N} X_{z_{i}}\big))\cap \Omega.$$
 Let us show that  $X_{1}$ is a closed analytic subset in $\Omega$. 
Consider $z \in \Omega$. If $z$ is in $Z_{\alpha}$ we have  $X_{1} = X$  in a neighbourhood of $z$ and the assertion is clear. If not, either $z$  is not in any $\partial U'_{i}$  and $X_{1}$ is the union of the $X_{z_{i}}$ in a neighbourhood of $z$ and the assertion is clear, or $z$ is in $\partial U'_{j_{1}}, \dots, \partial U'_{j_{k}}$ for $j_{1}, \dots, j_{k}$ in $[1,N]$. As the set  $X_{j_{h}}$ is closed and analytic in $Z_{\alpha}\cup U_{j_{h}}$, then $X_{1}$ is again the union of the $X_{z_{i}}$  near $z$ in $\Omega$, and the assertion is proved.

 So in this situation there exists a real  positive  $\beta < \alpha$ such that  $Z_{\beta}\subset \Omega$. Let  $X_{2}$  be the irreducible component of  $X_{1}\cap Z_{\beta}$  which contains $X$; then $X_{2}$  is a closed irreducible analytic subset of $Z_{\beta}$ such that $X_{2} \cap Z_{\alpha} = X$.

 Now let
 $$ \gamma := \inf\{ \beta \leq \alpha \ / \ \exists X_{\beta} \ \mbox{irreducible  $n$-cycle  of $Z_{\beta}$ such that $X_{\beta}\cap Z_{\alpha} = X$} \}.$$
 Then what we obtained above shows that we have $\gamma < \alpha$, and, applying the same arguments to the cycle $X_{\gamma}$ defined on $Z_{\gamma}$ via the cover of $Z_{\gamma}$  by the $Z_{\beta}, \beta > \gamma$ in which we already built an irreducible  $n$-cycle  $X_{\beta}$ extending $X$, we conclude that  $\gamma = 0$  and that there exists an (unique) irreducible $n$-cycle  $\tilde{X}$  in $Z$ extending $X$.
\hfill $\Box$
\end{proof}

\subsection{Some consequences.}

We shall give first  some easy consequences of the fact that the reduced complex space $Z$ is  strongly $n$-concave.

\begin{prop}\label{type fini}
Let  $n \geq 2$ be an integer and let $Z$ be a reduced complex space which is  strongly $(n-2)$-concave. Then the natural map $ j : \mathcal{C}_{n}^{\mathsf{f}}(Z) \to \mathcal{C}_{n}^{loc}(Z)$ is a homeomorphism. Moreover, for each   $\alpha \in ]0,1[$  the restriction map
 $$res_{\alpha} : \mathcal{C}_{n}^{\mathsf{f}}(Z) \to \mathcal{C}_{n}^{\mathsf{f}}(Z_{\alpha})$$
 is well defined and is also a homeomorphism.
 \end{prop}

\begin{proof}
Let us prove first that   any $n$-cycle $X$ in $Z$ has finitely many irreducible components. As this implies the same result for each $Z_{\alpha}$ for $\alpha \in ]0, 1[$, this will imply the fact that the restriction map $res_{\alpha}$ is well defined,  and then bijective as a consequence of Theorem \ref{prlgt global}.

 As the family of  irreducible components of a $n$-cycle is  locally finite, only finitely many irreducible components of $X$  can  meet the compact set $K := \varphi^{-1}([1,2])$. But we have seen in the remark following the previous theorem that {\bf any} irreducible component of $X$ must meet $K$. So $X$ is a finite type cycle.

 To show the continuity of $res_{\alpha}^{-1}$ it is then enough to prove that $j$ is a homeomorphism which is an easy consequence of the lemma below.

\begin{lemma}\label{petit lemme}
Let $Z$ be a reduced complex space and let $(X_{\nu})_{\nu \geq 0}$ be a sequence of $n$-cycles in $Z$ converging in $\mathcal{C}_{n}^{loc}(Z)$ to a cycle $Y$. Assume that there exists a relatively compact open set $\Omega$ in $Z$ such that any irreducible component of each $X_{\nu}$ and of  $Y$ meets  $\Omega$. Then all these cycles are of finite type and the sequence converges to $Y$ in $\mathcal{C}_{n}^{\mathsf{f}}(Z)$.
\end{lemma}

\begin{proof}[Proof of Lemma \ref{petit lemme}] 
First if $Y = \emptyset$, for $\nu \gg 1$ the cycle $X_{\nu}$ will be disjoint from the compact set $K := \bar \Omega$, and this implies that $X_{\nu}$ is the empty cycle. So the conclusion holds in this case. If  $Y$ is not empty, let $U$ be a relatively compact  open set in $Z$ meeting each irreducible component of $Y$. We have to show, by definition of the topology of $\mathcal{C}_{n}^{\mathsf{f}}(Z) $, that for  $\nu \gg 1$ each irreducible component of  $X_{\nu}$ meets $U$. If it is not the case,  passing to a subsequence, we may assume that for each $\nu$ there exists an irreducible component $\Gamma_{\nu}$ of $X_{\nu}$ disjoint from $U$. As, passing again to a subsequence, we may assume that the sequence  $(\Gamma_{\nu})$ converges in $\mathcal{C}_{n}^{loc}(Z)$ to a cycle $\Gamma$, we shall have $\vert \Gamma \vert \subset \vert Y\vert$ and $\vert \Gamma\vert \cap U = \emptyset$. To conclude, it is enough to show that  $\Gamma$ is not the empty cycle, as any irreducible component of $\Gamma$ is also an irreducible component  of $Y$ and then meets $U$ by hypothesis. As each $\Gamma_{\nu}$ is not empty, it has to meet $K = \bar \Omega$. This implies that $\Gamma$ also meets $K$ and so is not empty. This contradicts our assumption.
\hfill $\Box$
\end{proof}

\noindent 
{\it End of the proof of Proposition \ref{type fini}.} We have proved that $j$, and then also each $j_{\alpha} : \mathcal{C}_{n}^{\mathsf{f}}(Z_{\alpha}) \to \mathcal{C}_{n}^{loc}(Z_{\alpha})$ for $\alpha \in ]0,1[$, is a holomorphic homeomorphism. To conlude the proof we have to show the continuity of  $res_{\alpha}^{-1}$, and this reduces to prove that if the sequence $(X_{\nu})$ of $\mathcal{C}_{n}^{loc}(Z)$ is such that the sequence $(X_{\nu}\cap Z_{\alpha})$ converges in $\mathcal{C}_{n}^{loc}(Z_{\alpha})$, then  it converges in $\mathcal{C}_{n}^{loc}(Z)$. Let $Y_{\alpha} \in \mathcal{C}_{n}^{loc}(Z_{\alpha})$ be the limit of this sequence in $\mathcal{C}_{n}^{loc}(Z_{\alpha})$ and let $Y \in \mathcal{C}_{n}^{loc}(Z)$  be the cycle extending it. Let $A$ be the set of $\beta \in ]0, \alpha]$ such that the sequence  $(X_{\nu}\cap Z_{\beta})$  converges in $ \mathcal{C}_{n}^{loc}(Z_{\beta})$ to  $Y\cap Z_{\beta} $. Then $\alpha$ is in $A$ so $A$ is not empty. Put $\gamma := \inf A$. Theorem \ref{Th. de finitude} implies that $\gamma = 0$ and we obtain also the convergence in any $\mathcal{C}_{n}^{loc}(Z_{\beta})$, for any $\beta > 0$; this gives the convergence in $\mathcal{C}_{n}^{loc}(Z)$, as, by definition, a $n$-scale on $Z$ is also a $n$-scale on $Z_{\beta}$ for $\beta > 0$ small enough.
\hfill $\Box$
\end{proof}

\subsection{An analytic extension criterion.}

The aim of this paragraph is to prove the following analytic extension result.

\begin{thm}\label{ncycles}
Let $Z$ be a complex space and $n$ an integer. Consider a $\mathsf{f}$-continuous\footnote{This means that we have a continuous family of finite type $n$-cycles such that its graph is quasi-proper over $S$. This is equivalent to the continuity of the classifying map $ \varphi : S \to \mathcal{C}_n^\mathsf{f}(Z)$ of this family.} family $(X_{s})_{s\in S}$  of $n$-cycles of finite type in $Z$ parametrized by a reduced complex space $S$. Fix a point  $s_{0}$ in $S$ and assume that the open set $Z'$ in $Z$ meets all irreducible components of $X_{s_{0}} $ and such that the family of cycles $(X_{s}\cap Z')_{s \in S}$ is analytic at  $s_{0}$. Then there exists an open neighbourhood  $S_{0}$ of $s_{0}$ in  $S$ such that the family  $(X_{s})_{s\in S_{0}}$ is $\mathsf{f}$-analytic.
\end{thm}

The hypotheses translated in terms of classifying maps means that we have a continuous map $\varphi : S \to \mathcal{C}_{n}^{\mathsf{f}}(Z) $ such that the composed map $r\circ\varphi$  is holomorphic at $s_{0}$, where $r : \mathcal{C}_{n}^{\mathsf{f}}(Z)  \to \mathcal{C}_{n}^{loc}(Z') $ is the restriction map.

Then the theorem says that there exists an open neighbourhood  $S_{0}$ of $s_{0}$ in  $S$ such that the map  $\varphi$ is holomorphic on $S_{0}$. Note that, as $r$ is holomorphic\footnote{in the sense that for any holomorphic map  $\psi : T \to \mathcal{C}_{n}^{\mathsf{f}}(Z) $ of a reduced complex space $T$ the composed map $r\circ\psi$ is holomorphic.}, the hypothesis  that $\varphi$ is holomorphic at $s_{0}$ is a necessary condition.

\medskip

This result is not true in general if we take for  $S$  a non smooth  Banach analytic set which is not of finite dimension (locally). The reader may find a counter-exemple with an isolated singularity in \cite[Chapter V]{BM2}.

The key point for the proof of the previous theorem is the following analytic extension result.

\begin{prop}\label{nfonction}
Let $S$ a reduced complex space and let  $\emptyset \not= U_{1} \subset U_{2}$  be two polydiscs in  $\C^{n}$. Let  $f : S \times U_{2} \to \C$ a continuous function, holomorphic on  $\{s\}\times U_{2}$  for each   $s \in S$. Assume moreover that the restriction of $f$ to  $S \times U_{1}$ is holomorphic. Then $f$ is holomorphic on   $S \times U_{2}$.
\end{prop}

\begin{proof}[Proof of Proposition \ref{nfonction}] 
First consider the case $S$ smooth. As the question is local on $S$ it is enough to consider the case where $S$ is an open set in $\C^{m}$. Fix an open relatively compact  polydisc $P$ in $S$. The function $f$ defines a map   $ F : U_{2} \to \mathscr{C}^{0}(\bar P, \C)$, where we write down $ \mathscr{C}^{0}(\bar P, \C)$ the Banach space of continuous functions on  $\bar P$, via the formula  $F(t)[s] = f(s,t)$ for $t \in U_{2}$ and  $s \in \bar P$. First we shall show that the map $F$ is holomorphic.

Let  $U \subset\subset U_{2} $ be a polydisc. For  any fix $s \in S$ we have
$$ \frac{\partial f}{\partial t_{i}}(s, t) = \frac{1}{(2i\pi)^{n}}.\int_{\partial\partial U} f(s,\tau).\frac{d\tau_{1}\wedge \dots \wedge d\tau_{n}}{(\tau_{1} - t_{1})\dots (\tau_{i}-t_{i})^{2}\dots (\tau_{n}-t_{n})} \quad \forall t \in U \quad \forall i \in [1,n].$$
where  $t := (t_{1}, \dots, t_{n})$ are coordinates on $\C^{n}$. This Cauchy formula shows that  $F$ is $\C$-differentiable and its differential at the point $t \in U$ is given by
$$h \mapsto \sum_{i=1}^{n} F_{i}(t).h_{i}, \quad h \in \C^{n},$$
where  $F_{i} $  is associated to the function
$$ (s,t) \mapsto  \frac{\partial f}{\partial t_{i}}(s, t) \quad i \in [1,n]$$
which is holomorphic for each fixed  $s \in S$  thanks to the Cauchy formula above.

Let  $H(\bar P,\C)$ the closed subspace of $\mathscr{C}^{0}(\bar P, \C)$ of functions which are holomorphic on $P$. Our hypothesis implies that the restriction of $F$ to the non empty open set $U_{1}$ takes its values in this subspace. Let us show that this is also true on $U_{2}$. Assume that there exists $t_{0} \in U_{2}$ such that $F(t_{0})$ is not in  $ H(\bar P,\C)$. Thanks to the {\it Hahn-Banach theorem} we can find a continuous linear form  $\lambda$  on $\mathscr{C}^{0}(\bar P, \C)$, vanishing on $H(\bar P,\C)$, and such that  $\lambda(F(t_{0})) \not= 0$. But the function $t \mapsto \lambda(F(t))$ is holomorphic on $U_{2}$ and vanishes on $U_{1}$; this contradicts  $\lambda(F(t_{0})) \not= 0$. So  $F$ takes values in $H(\bar P,\C)$ and $f$ is holomorphic on $S\times U_{2}$ when $S$ is a complex manifold.

The case where $S$ is a weakly normal complex space follows immediately.

When $S$ is a general reduced complex space, the function $f$ is meromorphic and continuous on $S\times U_{2}$ and holomorphic on $S\times U_{1}$. So the closed analytic subset $Y \subset S \times U_{2}$ along which $f$ may not be holomorphic has no interior point in each $\{s\}\times U_{2}$. The analytic extension criterion of \cite[Chapter IV, Criterion 3.1.7]{BM1} allows to conclude.
\hfill $\Box$
\end{proof}

\begin{proof}[Proof of Theorem \ref{ncycles}]
Let $\vert G\vert \subset S \times Z$  be the graph of the $f$-continuous family $(X_{s})_{s\in S}$ and let $A$ be the set of points in
$(\sigma,\zeta) \in \vert G\vert$  admitting an open neighbourhood $S_{\sigma}\times Z_{\zeta}$ in $S \times Z$ such that the family of cycles $(X_{s} \cap Z_{\zeta})_{s \in S_{\sigma}}$ is analytic. Remark that, thanks to our hypothesis, the open set $A$ in $\vert G\vert$   meets every irreducible component of  $\{s_{0}\}\times \vert X_{s_{0}}\vert$.

Assume to begin that there exists a smooth point $z_{0}$ of  $\vert X_{s_{0}}\vert$ in the boundary of $A \cap (\{s_{0}\}\times \vert X_{s_{0}}\vert)$.  Choose a $n$-scale $E := (U,B,j)$   on $Z$ which is adapted to $X_{s_{0}}$ and satisfying :
 \begin{align*}
 & \deg_{E}(\vert X_{s_{0}}\vert) = 1, \quad   j_{*}(X_{s_{0}}) = k.(U \times \{0\})\\
  & z_{0} \in j^{-1}(U\times B),  \quad    j(z_{0}) := (t_{0}, 0).
 \end{align*}
It is clear that such a $n$-scale exists as $z_{0}$ is a smooth point in $\vert X_{s_{0}}\vert$. Let  $S_{1}$ be a sufficiently small open neighbourhood of $s_{0}$ in  $S$ and let   $f : S_{1}\times U \to \Sym^{k}(B)$ be the (continuous) classifying map for the family $(X_{s})_{s\in S_{1}}$ in the scale $E$. As $j^{-1}(U\times \{0\})$ meets $A$, there exists a non empty polydisc  $U_{2} \subset U$ such that   $U_{2}\times \{0\}$ is contained in $A$. Then we may apply Proposition \ref{nfonction} to each scalar  component of $f$ in order to obtain that $f$ is holomorphic on $S_{1}\times U$. Moreover, as the same argument applies to any linear projection of $U\times B$ to $U$ near enough the vertical one; this implies that $f$ is an isotropic map, up to shrinking slightly $U$. This contradicts the fact that the point $(s_{0}, z_{0})$ is in the boundary of the open set  $A \cap (\{s_{0}\}\times \vert X_{s_{0}}\vert)$ of  $ \vert X_{s_{0}}\vert$.

If the boundary of   $A \cap (\{s_{0}\}\times \vert X_{s_{0}}\vert)$ is contained in the singular set of  $\vert X_{s_{0}}\vert$, we may apply the analytic extension criterion of   \cite[Chapter IV, Criterion 3.1.7]{BM1}, and we obtain directly that $A$ contains $\vert X_{s_{0}}\vert$.  So in any case the family of cycles $(X_{s})_{s\in S}$ is analytic at $s_{0}$. As the graph $\vert G\vert$ is, by assumption, quasi-proper on $S$, it is enough to use the next proposition (which is proved in \cite[Proposition 2.2.3]{B15})  to conclude.
\end{proof}

 \begin{prop}\label{open}
 Let $Z $ and $S$ be reduced complex spaces and let  $(X_{s})_{s \in S}$  be a f-continuous family of $n$-cycles in $Z$. Assume that this family is analytic in $s_{0} \in S$. Then there exists an open neighbourhood  $S'$ of $s_{0}$ in $S$  such that the family $(X_{s})_{s \in S'}$ is a  $\mathsf{f}$-analytic family of $n$-cycles in $Z$.
  \end{prop}

\subsection{Proof of Theorem \ref{global} and its corollary}

We shall begin by a lemma which will give the case where the $n$-cycle $X_{0}$ is compact.

 \begin{lemma}\label{loc = f}
 Let $Z$ be a strongly $(n-2)$-concave reduced complex space. Then $\mathcal{C}_{n}(Z)$ is open  in $\mathcal{C}_{n}^{\mathsf{f}}(Z)$.
 \end{lemma}

\begin{proof} Let $X_{0}$ be a compact cycle in $Z$. There exists $\alpha \in ]0, 1[$ such that $X_{0}$ is contained in $Z_{\alpha} = \{x \in Z \ / \  \varphi(x) > \alpha \}$. So $X_{0}$ does not meet the compact set $\varphi^{-1}(\{\alpha\})$. This is an open condition in $\mathcal{C}_{n}^{\mathsf{f}}(Z)$. And as any irreducible $n$-dimensional analytic subset in $Z$ has to meet $K$, if it does not meet $\varphi^{-1}(\{\alpha\})$ it is contained in $Z_{\alpha}$ (by connectedness). Then any $X \in \mathcal{C}_{n}^{\mathsf{f}}(Z)$ which is near enough $X_{0}$ is contained in $Z_{\alpha}$ so is compact.
\hfill $\Box$
\end{proof}

 Note that under the hypothesis of the previous lemma, $\mathcal{C}_{n}(Z)$ is not closed in $\mathcal{C}_{n}^{\mathsf{f}}(Z)$  in general, as one can see taking $Z := \mathbb{P}_{N}\setminus \{0\}$ and considering the set  of hyperplanes in $\mathbb{P}_{N}$.

\medskip

 As we already know from \cite{B75} that  $\mathcal{C}_{n}(Z)$ is a reduced complex space, Theorem \ref{global} and its corollary are proved near a compact $n$-cycle in $Z$.

\medskip

\begin{proof}[The case  where $X_{0}$ is not compact] 
Of course we are in the case where $Z$ is not compact.  Fix $X_{0}$ a non compact $n$-cycle in $\mathcal{C}_{n}^{\mathsf{f}}(Z)$ and choose an $\alpha \in ]0, 1[$ which is not a critical value of  $\varphi$ and of the restriction of $\varphi$ to  $\vert X_{0}\vert$; this is possible thanks to Sard's Lemma  \ref{Sard} and the fact that $\varphi(\vert X_{0}\vert)$ contains $]0, 1[$.

 Consider now the reduced complex space $\Xi$ constructed in Theorem \ref{Th. de finitude}. It is an open neighbourhood of $X_{0}\cap \Delta''$ in $\mathcal{C}_{n}^{\mathsf{f}}(\Delta'')$ and we may assume that $\Delta'' := Z_{\beta}$ for some $\beta < \alpha$ very near $\alpha$. Due to Proposition \ref{type fini} it is homeomorphic to an open neighbourhood $\mathcal{V}$ of $X_{0}$ in $\mathcal{C}_{n}^{\mathsf{f}}(Z)$. So the restriction map
 $$ res_{0} : \mathcal{V} \to \Xi $$
 is holomorphic, bijective and is a  homeomorphism. Now the continuity of $res_{0}^{-1}$ and the finiteness of $\Xi$ allow to apply Theorem \ref{ncycles} because we already know that the tautological family of cycles parametrized by $\Xi$ is \textsf{f}-analytic on the open set $\Delta''$ which is an open set which meets any irreducible component of each cycle in this family (because $K \subset Z_{\beta}$). Then $res_{0}^{-1}$ is holomorphic and so $res_{0}$ is an isomorphism of Banach analytic sets.
\hfill $\Box$
\end{proof}

  \subsection{A compactness criterion for the connected components of the reduced complex space  $\mathcal{C}_{n}^{\mathsf{f}}(Z)$ when $Z$ is strongly $(n-2)$-concave}

  For a reduced complex space $Z$ which is compact, the compactness of the connected components of $\mathcal{C}_{n}(Z)$ is a consequence of the existence of a  $\mathscr{C}^{1}$ $2n$-differential form on $Z$ which is $d$-closed and such that its $(n, n)$ part is positive definite in the Lelong sense. Indeed, this gives that the volume (computed with this $(n, n)$ part) of the $n$-cycles is constant on connected components. The result follows then from Bishop's theorem \cite{Bis64} (see \cite[Chapter IV]{BM1} for details).

  In the case of a non compact strongly $(n-2)$-concave reduced complex space $Z$ we have the following analogous result :

  \begin{prop}\label{compacite}
  Let $Z$ be a reduced complex space which is strongly $(n-2)$-concave. Assume that there exists on $Z$  a $\mathscr{C}^{1}$ $2n$-differential form  $\omega$  which is $d$-closed with compact support and such that its $(n, n)$ part is positive definite in the Lelong sense in a neighbourhood of $K := \varphi^{-1}([1, 2])$, and everywhere non negative in the Lelong sense. Then the connected components of $\mathcal{C}_{n}^{\mathsf{f}}(Z)$ are compact.
   \end{prop}

\begin{proof}
For $\alpha < 1$ near enough to $1$ and  for any continuous hermitian metric $h$ on $Z$ there exists a constant $C$ such that the following inequality holds:
    $$ vol_{h}(X\cap Z_{\alpha}) \leq C.\int_{X} \ \omega \quad \quad {\rm for \ any \ cycle} \ X \in \mathcal{C}_{n}^{\mathsf{f}}(Z).$$
    As the function $X \mapsto \int_{X} \omega$ is locally constant on  $\mathcal{C}_{n}^{\mathsf{f}}(Z)$ because $d\omega = 0$ (the direct image of $\omega$ as a current is $d$-closed, so locally constant at smooth points of $\mathcal{C}^{\mathsf{f}}_{n}(Z)$, and this current is a continuous function on $\mathcal{C}^{\mathsf{f}}_{n}(Z)$ thanks to Proposition IV 2.3.1 of {\it loc. cit.}), we have a uniform bound for the volume of $X \cap Z_{\alpha}$ for $X$  in a given connected component of  $\mathcal{C}_{n}^{\mathsf{f}}(Z)$. This implies that the closure of the image of this connected component in $\mathcal{C}_{n}^{\mathsf{f}}(Z_{\alpha})$ is compact, thanks to Bishop's theorem (see   \cite[Chapter IV, Theorem 2.7.20]{BM1}).
     But the restriction map $\mathcal{C}_{n}^{\mathsf{f}}(Z) \to \mathcal{C}_{n}^{\mathsf{f}}(Z_{\alpha})$ is a homeomorphism by Proposition \ref{type fini}, so the image of a connected component is closed and then compact.
\hfill $\Box$
\end{proof}

\providecommand{\bysame}{\leavevmode\hbox to3em{\hrulefill}\thinspace}
%
%

\bibliographystyle{amsalpha}
\bibliographymark{References}

\end{document}